\documentclass[11pt]{article}
 \usepackage{amsfonts}
 \usepackage{amsmath}
\usepackage[numbers]{natbib}
\usepackage{csquotes}

\usepackage[left=2.0cm,top=3cm,right=2.0cm,bottom= 5cm]{geometry}

\usepackage[usenames,dvipsnames]{color}

\newtheorem{theorem}{Theorem}[section]
\newtheorem{definition}[theorem]{Definition}
\newtheorem{proposition}[theorem]{Proposition}
\newtheorem{conjecture}[theorem]{Conjecture}
\newtheorem{corollary}[theorem]{Corollary}
\newtheorem{lemma}[theorem]{Lemma}
\newtheorem{remark}[theorem]{Remark}
\newtheorem{example}[theorem]{Example}
\newtheorem{examples}[theorem]{Examples}
\newtheorem{foo}[theorem]{Remarks}

\newenvironment{proof}{\addvspace{\medskipamount}\par\noindent{\it Proof}.}
{\unskip\nobreak\hfill$\Box$\par\addvspace{\medskipamount}}

\newcommand{\essinf}{\mathop{\rm ess\,inf}}
\newcommand{\esssup}{\mathop{\rm ess\,sup}}

\newcommand{\E}[1]{\mathbb{E} \left[#1\right]}

\title{ On the dynamic representation of some time-inconsistent risk measures in a Brownian filtration}

\author{Julio Backhoff Veraguas\footnote{Vienna University of Technology, Institute of Statistics and Mathematical Methods in Economics (E105-7), Wiedner Hauptstraße 8-10
A-1040 Vienna, Austria (email:julio.backhoff@tuwien.ac.at, +435880110575)}\and Ludovic Tangpi\footnote{Universit\"at Wien, Fakult\"at f\"ur Mathematik, Oskar-Morgenstern-Platz 1, A-1090 Wien, Austria (email:ludovic.tangpi@univie.ac.at, +431427750447) Financial support from Vienna Science and Technology Fund (WWTF) under Grant MA 14-008 is gratefully acknowledged.}}

\begin{document}

\maketitle

\abstract{It is well-known from the work of Kupper and Schachermayer that most law-invariant risk measures do not admit a time-consistent representation.
In this work we show that in a Brownian filtration the ``Optimized Certainty Equivalent'' risk measures of Ben-Tal and Teboulle
can be computed through PDE techniques, i.e.\ dynamically. This can be seen as a substitute of sorts whenever they lack time consistency, and covers the cases of  conditional value-at-risk and monotone mean-variance. 
Our method consists of focusing on the convex dual representation, which suggests extending the state space. With this we can obtain a dynamic programming principle and use stochastic control techniques, along with the theory of viscosity solutions, which we must adapt to cover the present  singular situation.
\vspace{.2cm}

\textbf{MSC 2010}: 93E20, 91G80, 60H30, 49N10, 35Q93, 35D40.

\textbf{Keywords}: Time-inconsistency, risk measures, optimized certainty equivalent, HJB equation, viscosity solution, unbounded stochastic control problem, dynamic programming principle, singular Hamiltonian.
  }

\section{Introduction}

Let $(\Omega, {\cal F}, P)$ be a probability space equipped with the completed filtration $({\cal F}_t)_{t \in [0,T]}$ of a $d$-dimensional Brownian motion $W$. We assume $T\in \mathbb{R}_+$ and ${\cal F}={\cal F}_T$. A functional $\rho: L^\infty({\cal F}) \to (-\infty, +\infty]$ that is convex, increasing and cash-invariant\footnote{$\rho(X+ c) = \rho(X) +c$ for all $X \in L^\infty({\cal F})$ and $c \in \mathbb{R} $. Translation invariance is a synonym for this.} is called convex risk measure\footnote{In fact, it is  $\tilde{\rho}(X):= \rho(-X)$ that satisfies the risk measures axioms as developed in \citet{artzner01} and \citet{FS3dr}, but we will work with the increasing functional $\rho$ for ease of notation.}.
The simplest example of risk measure is the mathematical expectation $\rho(X):= E[X]$.
By the martingale representation theorem, it satisfies the dynamic representation $E[X] =  X - \int_0^TZ_u\,dW_u$ where $Z$ is a $W$-integrable process.
Deriving such dynamic representations for general risk measures has given rise to a vast literature, mainly because they provide insight into the structure of the risk measure itself and due to their potential relevance in applications (for instance, in dealing with stochastic control problems as in \citet{rouge}).
Such dynamic representations are well understood in the case where $\rho$ stems from a dynamic convex risk measure, that is $\rho = \rho_{0,T}$ and
 $(\rho_{\nu,\tau})_{0\le \nu\le\tau\le T}$ (with $\nu,\tau$ being stopping times) is a family of functionals $\rho_{\nu,\tau}: L^\infty({\cal F}_\tau) \to L^\infty({\cal F}_\nu)$.
The main condition under which a dynamic representation can be derived is time-consistency\footnote{This condition is also known as the flow property.
We refer to \citet{cheridito01,delbaen05,artzner02,RS,defsca,foellmerpenner} for discussions on the consequences of time-consistency.}, which amounts to 
$$\rho_{\sigma, \nu}(X)=\rho_{\sigma,\tau}(\rho_{\tau,\nu}(X))$$ for all stopping times $0\le \sigma\le \tau\le \nu\le T$ and $X \in L^\infty({\cal F}_\nu)$. In fact, in this case, for every $X \in L^\infty$ and every $t \in [0,T]$ one has $\rho_{t,T}(X) = Y_t$ where $(Y,Z)$ is the unique (minimal super-)solution of the backward stochastic differential equation\footnote{Throughout the paper, equalities and other ``pointwise'' relations are understood in the $P$-a.s.\ sense.} (BSDE)
\begin{equation}\textstyle
\label{eq:bsde}
	Y_t = X + \int_t^Tg_u(Z_u)\,du - \int_t^TZ_u\,dW_u\,,
\end{equation}
for a given function $g:[0,T]\times \Omega\times \mathbb{R}^d \to \mathbb{R}$,
see \citet{memin02}, \citet{Del-Peng-Ros} and \citet{tarpodual}.
Using the well-established link between BSDE and partial differential equations, these representations show that in the Markovian setting, $\rho_{t,T}$ can be written as the viscosity solution, or the minimal viscosity supersolution of a non-linear PDE, see \citet{karoui01} and \citet{Dra-Mai} respectively.

Time-consistency plays a crucial role in the aforementioned dynamic representation results.
However, as shown by \citet{kupper02}, most commonly used (law invariant) risk measures, such as the conditional value-at-risk (also referred to as tail or average value-at-risk), do not enjoy this property. Notable exceptions are the expected value and the so-called entropic risk measure. The declared aim of this paper is to show that nevertheless many interesting time-inconsistent risk measures can be computed dynamically. {This is achieved by first establishing a relevant dynamic programming principle in an enlarged state space, and then through its infinitesimal counterpart: non-linear PDEs. In our opinion, this serves as a replacement of sorts for the lack of time-consistency.}

In this work, we focus on the case of so-called optimized certainty equivalent (OCE) risk measures; see \citet{ben-tal00,ben-tal01}. This is a class containing time-consistent risk measures such as the entropic one, as well as time-inconsistent ones such as the conditional value-at-risk of \citet{Rock-Uryasev}, the monotone mean-variance of \citet{MMR2006}, and more generally risk measures with power-type penalty functions (e.g.\ as the R\'enyi divergence).
This covers some of the most relevant measures of risk available. For instance, the conditional value-at-risk (also known as expected shortfall) has been praised, and its adoption recommended, by the Basel III Committee in the following terms%
\begin{displayquote}
``... the current framework's reliance on VaR (value-at-risk) as a quantitative risk metric raises a 
number  of  issues,  most  notably  the  inability  of  the  measure  to  capture  the  “tail  risk”  of  the  loss  
distribution.  The  Committee  has  therefore  decided  to  use  an  expected  shortfall  (ES)  measure  for  the  
internal  models
-based  approach  and  will  determine  the  risk  weights  for  the 
revised 
standardised 
approach  using  an  ES  methodology...''
\end{displayquote}
see \cite[Page 18]{Basel}.
\vspace{3pt}

{In Section \ref{sec: non OCE} we shall further study the extension of our approach outside this class, namely the to so-called ``utility-based expected shortfall'' of \citet{FS3dr}, and we show that our results do not fully carry over to this class of risk measures.} Regarding the claims, i.e.\ the random variables, whose risk we aim to evaluate/compute, we shall be mainly concerned with what we call ``Markovian claims.'' These are bounded random variables of the form 
\begin{equation}\textstyle
X=f(Y_T)+\int_0^T g(t,Y_t)dt,\label{eq markov claim}
\end{equation} 
where now $Y$ denotes an It\^o-diffusion. We can think of such claims as (limits of) static positions written on a diffusion model. In Section \ref{sec: heuristic general claims} we shall describe how our method can be adapted to accommodate more general claims.

The main result in this article is Theorem \ref{thm:exist}. It states that for most OCE risk measures, the risk of a Markovian claim (namely $\rho(X)$) can be computed dynamically as the initial value of a non-linear Hamilton-Jacobi-Bellman (HJB) partial differential equation going backwards in time. 
The road leading to this result starts with the dual representation of the risk measure and a simple ``enlargement of state space'' idea which allows to interpret the evaluation of the risk measure as a stochastic optimal control problem of its own. 
This point of view allows to obtain a suitable dynamic programming principle (DPP) in the mentioned enlarged state space; see Corollary \ref{cor:bellman}. We stress that we can obtain the DPP without going through the typicall technical hurdles associated to it, by profiting from the specific form of the risk measures we analyze\footnote{In fact a DPP of sorts holds in greater generality than we need for the applications in this article; see Proposition \ref{prop: Bellman gnral}.}. As usual in stochastic control theory, we leverage on this DPP to obtain the aforementioned HJB equation.
This equation characterizes the value function of the named stochastic control problem as its (minimal) viscosity solution. Both the stochastic control problem and the HJB equation are in principle very intractable and degenerate; for instance the associated Hamiltonian may explode (one says the problem/Hamiltonian is singular). This gives rise to most technical difficulties we encounter, and our efforts are largely devoted to dealing with them through several approximations and reductions.

For our main result, we have in mind the risk evaluation of claims as in \eqref{eq markov claim}, where $f,g$ are continuous but otherwise rough. This is motivated by financial applications, since most interesting (vanilla) options are functions with kinks evaluated on the underlying price process. This makes necessary the approach with viscosity solutions just described, as we show via examples. We will nevertheless explore the question of smoothness and existence of classical solutions for our HJB equation in Section \ref{sec examples}. There we focus on concrete OCE risk measures and make all necessary smoothness assumptions on the data of the problem.

There already exists a body of literature on 
the efficient handling of time-inconsistency in the framework of risk measures. 
We refer the reader to \citet{PflugPichlerinconsistency1,PflugPichlerinconsistency2,
Bauerleinconsistent,CTMP} for a discrete-time set-up and \citet{MillerYang,Kar-Ma-Zha,Oksendaltrivial} and the references therein for a continuous-time one.
Observe that these  articles go beyond risk evaluation and consider decision making (i.e.\ risk minimization) on top of that, whereas the present work is concerned with dynamic representations alone. We see this as a necessary and challenging first step, and we will address the actual risk minimization problem in a follow-up work. 
The idea of ``enlarging the state space'' is also present in the discrete-time formulations, whereas HJB equations also appear in the continuous-time setting of \citet{MillerYang}, albeit employed in a very different way. The article closest to ours from a methodological point of view is \citet{Oksendaltrivial}; the main difference is that the authors work in a jump-diffusion setting and start by assuming existence of classical solutions (as opposed to viscosity ones). 
For convenience of the reader, we sketch in Section \ref{subsection Jump case} the mentioned jump-diffusion setting, but we leave open the rigorous treatment of the associated non-local HJB from a viscosity perspective; we expect that similar but more involved arguments as in the Brownian setting are applicable here. Other approaches to time-inconsistency can be found in e.g.\ \citet{Shapiro-letter,EkelandLazrak,ZhouLi}.  

As we have observed, the stochastic control representation we obtain involves a singular Hamiltonian. We refer to \citet{DaLio06,Dalio08}, \citet[Chap. 4]{Pham}, and references therein for results in this direction. Our stochastic control problem does not have the structure needed for these works, so we have to argue in a self-contained way; see Definition \ref{def viscosity} for the concept of viscosity solution we consider, and the discussion thereafter. A previous version of our work \cite{BT2016} uses the Stochastic-Perron method of \citet{Bayraktar-sir} to prove a milder version of our main result Theorem \ref{thm:exist}. {The current approach rests on the DPP, whose rigorous proof is seemingly direct owing to the structure of the problem at hands, and is therefore a more classical one. See the standard references \citet{Flem-Soner-second,YongZhou} for DPP in continuous time stochastic control, as well as \citet{Bou-Tou11} and \citet{ElKaroui-Tan13} for more recent developments.}

The article is organized as follows. In Section \ref{sec: setting and main results} we outline the setting of the problem, we provide and prove our main results (viscosity characterization and DPP). Then in Section \ref{sec examples} we provide examples and explore the issue of existence of classical solutions to our HJB equation. Section \ref{sec: extensions} is devoted to extending the applicability of our main result to more general claims, a broader class of risk measures, and the setting with jumps. Finally we provide some pending proofs in the appendix.

\section{PDE representation of Optimized Certainty Equivalents}
\label{sec: setting and main results}

\subsection{Setting and main result}
We call a convex function $l:\mathbb{R} \to \mathbb{R}$ \textit{loss function} if it is increasing, and satisfies $l(0) = 0$.
Every loss function is then continuous.
Denote $$l^\ast(z):=\sup_{x \in\mathbb{R} }\{xz - l(x)\}, \,\,\, z\ge 0,$$ the convex conjugate of $l$ and by $\text{dom}(l^*):= \{z\in \mathbb{R}_+: l^*(z)<\infty\}$ its domain.
Observe that necessarily $l^*\geq 0$. 
In what follows, we always assume that 
 the loss function satisfies the conditions
\begin{itemize}
\item[(N)]: $l^*(1) =0$. \label{norm}
\item[(C)]: $l(x)>x$ for all $x$ such that $|x|$ is large enough.
\end{itemize}
\begin{definition}
	The optimized certainty equivalent\footnote{This corresponds to the standard OCE risk measure up to a minus sign.} (OCE) associated to $l$ is the functional $\rho$ given by
	\begin{equation}\label{eq inf min}
		\rho(X):= \inf_{r \in \mathbb{R}}(E[l( X-r)]+r).
	\end{equation}
\end{definition}
Condition (N) above ensures that
$\rho(0) = 0$.
This is a normalization condition which is standard in risk measures theory, see for instance \citet{Del-Peng-Ros}, but of course mathematically non-essential. Condition (C) guarantees that the infimum in \eqref{eq inf  min} is attained and behaves stably, and is equivalent to the existence of $x_{-}<0<x_{+}$ s.t. $l(x_\pm)>x_\pm$. It also entails the non-emptiness of the interior of $\text{dom}(l^*)$. %
  This setting covers many risk measures which we will encounter in Section \ref{sec examples}, for instance the entropic one, the conditional value-at-risk (CVaR) and the monotone mean-variance, as well as many others. On the other hand, our assumptions rule out the risk neutral case $\rho(X)= E[X]$ of $l(x)= x$.

Let ${\cal O}$  be the interior of $\text{dom}(l^*)$, namely $${\cal O}:= \mbox{int}(\text{dom}(l^*)).$$ 
For every $(s, y) \in [0,T]\times \mathbb{R}^m $, we consider the It\^o diffusion $Y^{s,y}$ given by%
\begin{align*}
dY^{s,y}_t &= b(t,Y^{s,y}_t)dt+\sigma(t,Y^{s,y}_t)dW_t, \quad t\ge s\\
Y^{s,y}_s&= y
\end{align*}
for two given functions $b$ and $\sigma$.
Thus seen, the process $Y^{0,y}$ will be the ``underlying'' upon which claims are written. The claims we shall mostly deal with, and whose risk $\rho(X)$ we want to compute, are assumed to be of the following ``Markovian form'':
\begin{align}\label{def markovian claim}
\textstyle 
X=f(Y_T^{0,y})+ \int_0^Tg(t,Y^{0,y}_t)\,dt.
\end{align}
We will make the following assumptions on the functions $b$, $\sigma$, $f$ and $g$:
\begin{itemize}
	\item[(A1)]\label{A1}  $b:[0,T]\times \mathbb{R}^m\to\mathbb{R}^m$ and $\sigma:[0,T]\times \mathbb{R}^m\to\mathbb{R}^{m\times d}$ are continuously differentiable and there exist $k_1,\lambda_1\ge 0$ such that
	\begin{equation*}
		|b(t,y) - b(t',y')|+ |\sigma(t,y) - \sigma(t',y')|\le k_1(|y-y'|+|t-t'|) \quad \text{and } |b(t,y)|+|\sigma(t,y)|\le \lambda_1(1 + |y|),
	\end{equation*}
	for all $t ,t'\in [0,T]$ and $y,y'\in \mathbb{R}^m$.
\end{itemize}
\begin{itemize}
	\item[(A2)] $g:[0,T]\times \mathbb{R}^m\to\mathbb{R}$ and $f:\mathbb{R}^m\to\mathbb{R}$ are continuous and there exist $k_2,\lambda_2\ge 0$ such that
	\begin{equation*}
		|g(t,y) - g(t,y')|+ |f(y) - f(y')|\le k_2|y-y'| \quad \text{and } |g(t,y)|+|f(y)|\le \lambda_2,
	\end{equation*}
	for all $t \in [0,T]$ and $y,y'\in \mathbb{R}^m$.\label{A3}
\end{itemize}

Let us now describe the partial differential equation that will allow us to compute the risk of such claims. The reader eager to know where this PDE comes from, may consult Propositions \ref{lem PDE} and \ref{prop:struct-V} below; otherwise it suffices to say that the PDE arises from the stochastic control interpretation of the dual representation of OCEs when seen in an enlarged state space (the $z$ variable denoting a generic element there). We first fix some notation; let us introduce the function $$\psi(y,z):=f(y)z -l^*(z),$$ and the $\mathbb{R}\cup\{+\infty\}$-valued function\footnote{$A'$ is the transpose of $A$.} 
\begin{equation*}
	H(t, y,z, \gamma, \Gamma) := \langle B(t,y,z), \gamma\rangle + z g(t, y) + \frac{1}{2}\sup_{\beta \in \mathbb{R}^d}tr(AA'\Gamma),
\end{equation*}
defined on $[0,T]\times \mathbb{R}^m\times {\cal O}\times\mathbb{R}^{m+1}\times \mathbb{R}^{(m+1)\times (m+1)} $, where
  $$B(t,y,z):= (b(t, y), 0)', \text{ and } A(t,y,z,\beta): = (\sigma(t,y),z\beta)'.$$
\begin{definition}\label{def viscosity}
	A continuous function $v$ defined on $[0,T]\times\mathbb{R}^m\times {\cal O}$ is said to be a viscosity supersolution of the Hamilton-Jacobi-Bellman (HJB) equation
	\begin{align}
	\label{eq:hjb}
	\begin{cases}
	&\partial_tV + H(s, y,z, DV, D^2V) = 0,\quad (s,y,z)\in [0,T)\times \mathbb{R}^m\times {\cal O}\\
	&V(T,y,z) = \psi(y,z),\quad (y,z)\in \mathbb{R}^m\times {\cal O},
	\end{cases}
	\end{align}
	 if for all $x_0=(s_0,y_0,z_0)\in [0,T]\times\mathbb{R}^m\times {\cal O}$ and $\varphi \in C^2([0,T]\times\mathbb{R}^m\times {\cal O})$ such that $x_0$ is a local minimizer of $v-\varphi$ and $\varphi(x_0) = v(x_0)$, we have $v(x_0)\ge \psi(y_0,z_0)$ $s_0=T$, and otherwise
	\begin{equation*}
		\partial_t\varphi(x_0) + H(x_0, D\varphi(x_0), D^2\varphi(x_0)) \le 0.
	\end{equation*}

	A continuous function $v$ defined on $[0,T]\times\mathbb{R}^m\times {\cal O}$ is said to be a viscosity subsolution of \eqref{eq:hjb}
	 if for all $x_0=(s_0,y_0,z_0)\in [0,T]\times\mathbb{R}^m\times {\cal O}$ and $\varphi \in C^2([0,T]\times\mathbb{R}^m\times {\cal O})$ such that $x_0$ is a local maximizer of $v -\varphi$ and $\varphi(x_0) = v(x_0)$, we have $v(x_0)\le \psi(y_0,z_0)$ if $s_0=T$, and otherwise if further $(x_0, D\varphi(x_0), D^2\varphi(x_0))\in \text{intdom}(H)$ we have 
	\begin{equation*}
		\partial_t\varphi(x_0) + H(x_0, D\varphi(x_0), D^2\varphi(x_0)) \ge 0.
	\end{equation*}

	A function is a viscosity solution if it is both a viscosity sub- and supersolution, and a viscosity (super)solution $\bar{v}$ of \eqref{eq:hjb} is said to be minimal for a class of functions if for every viscosity (super)solution $v$ of \eqref{eq:hjb} in the given class one has $\bar{v}(s,y,z)\le v(s,y,z)$ for all $(s,y,z)\in [0,T]\times\mathbb{R}^m\times {\cal O}$.

\end{definition}
	Spelled out explicitly in terms of the data, the HJB equation \eqref{eq:hjb} takes the form
	\begin{align*}
		\partial_tV+b(s,y)\partial_yV+&\frac{1}{2}tr\left(\sigma(s,y)\sigma(s,y)'\partial^2_{yy}V\right)\\
		&+\sup_{\beta\in\mathbb{R}^d}\left [\frac{1}{2} z^2|\beta|^2\partial^2_{zz}V+ z \,\partial^2_{yz}V \sigma(s,y)\beta\right ] + zg(s,y) =0,
	\end{align*}
with
$$V(T,y,z)= f(y)z - l^*(z).$$
Thus, the Hamiltonian $H$ takes values on the extended real line, since the control space is unbounded (i.e.\ the problem/Hamiltonian is singular).
Our HJB equation \eqref{eq:hjb} is different from the standard HJB used in such settings, see e.g. \citet[Chap. 4]{Pham}. In the usual approach for singular Hamiltonians, one specifies a variational inequality with help of an auxiliary continuous function $G$ which signals/attests the points where the Hamiltonian explodes (typically $G\geq 0 \iff H<\infty$). However in our setting it is easy to see that there is no such continuous function. This justifies the definition of viscosity subsolutions we consider.
The cost to pay for this change of the definition is uniqueness, since there is no readily applicable comparison-principle type of result.
Of course, when $H$ is finite valued our definition coincides with the standard definition.

In the next subsection, we prove the following dynamic representation of OCEs, which is the main result of this article. It characterizes the value of an OCE as the initial value of the minimal solution of \eqref{eq:hjb}, which is uniquely determined of course. This is done for Markovian claims; {see however Section \ref{sec: heuristic general claims} for an extension of this approach.} 

\begin{theorem}
\label{thm:exist}%
	If (A1)-(A2) hold, then for $X^y:=f(Y_T^{0,y})+ \int_0^Tg(t,Y^{0,y}_t)\,dt$ we have $$\rho(X^y) = V(0,y,1),$$ where $V:[0,T]\times \mathbb{R}^m\times {\cal O}\to \mathbb{R}$ is a viscosity solution of the PDE \eqref{eq:hjb}. Furthermore if either  $\text{dom}(l^*)$ is bounded, or $l^*$ is finite and has polynomial growth, then  $V$ is the minimal supersolution of \eqref{eq:hjb} in the class of functions with polynomial growth.
\end{theorem}

To be precise, the proof of Theorem \ref{thm:exist} actually establishes the result for $V$ the value function of a stochastic control problem related to the dual representation of $\rho(X)$. See also Remark \ref{rem growth} for some comments on the growth properties of $V$. 
We emphasize that in principle the question of whether \eqref{eq:hjb} has a viscosity solution  is not amenable to standard methods. 
The main difficulty is the fact that $H$ is singular. On top of that, the function $A$ is not necessarily uniformly bounded nor uniformly Lipschitz. We refer to \citet{DaLio06,Dalio08} and \citet[Chap. 4]{Pham} for a discussion on some of these issues, but stress that our setting is not covered by the results therein.
This leads us to solve the PDE \eqref{eq:hjb} through several reductions and  approximations.
Uniqueness however remains an open problem, see Remark \ref{rem:unique} below.

At this point one could ask whether the approach through viscosity solutions is necessary. Indeed it is: Equation \eqref{eq:hjb} cannot be expected to have classical solutions in general, since $A$ is not uniformly parabolic. The following examples show that $y\mapsto \rho(X^y)$ in the above theorem is not necessarily a differentiable function, as soon as the claim is degenerate or more generally if the diffusion $Y$ is not ``uniformly parabolic''. This has consequences for $V$ of course, as we exemplify in Remark \ref{rem:roughness of V} below.
\begin{example}\label{ex preliminary} \leavevmode
\begin{itemize}
\item[1.] Let $\rho(X)= \log E [e^X]$ be the entropic risk measure, which is an OCE with loss function $l(x)=e^{x}-1$. Let $Y_t =y \in \mathbb{R}$ for all $t$ and $X^y :=f(Y_T)$ for $f$ Lipschitz but non-differentiable. Since $Y$ is deterministic we have $\rho(X^y)= f(y)$. This also holds for arbitrary OCEs. 
\item[2.] We now take $Y_T:=\text{sign}(W_T)+y$ and $X^y=Y_T^+$. From $Y_T$ it is not difficult to build the martingale diffusion $Y$ which ends up at $Y_T$ at time $T$, via the Markov property. The corresponding diffusion coefficient is
$$\textstyle \sigma(t,\cdot)=\sqrt{\frac{2}{\pi(T-t)}}\exp\left\{ -\frac{1}{2} \left  [  \Phi^{-1}\left (\frac{y+1-\cdot}{2}\right )\right ]^2 \right\},$$
where $\Phi$ is the distribution function of a standard Gaussian. In particular $\sigma$ is not uniformly parabolic. For the entropic risk measure we have
$$\textstyle \rho(X^y)=\log\left ( \frac{1}{2}\left (\exp\{[y-1]_+\}+\exp\{[y+1]_+\}\right ) \right ),$$ which is continuous but is not differentiable at $y=\pm 1$.
\item[3.] For $Y_t = W_t-y$, which is a uniformly parabolic model, and $X^y:= Y^+_T$, we have
$\rho(X^y) = \log(E[e^{Y^+_T}]) = \log\int e^{[c]_+} h(y+c)dc $, where $h$ is the density of a centred Gaussian with variance $T$. Thus $\rho(X^y)$ is smooth in $y$. This is the well-known smoothing effect of uniform-noise in action.
\end{itemize}
\end{example}

\begin{remark}
\label{rem:unique}{Even in the time-consistent case discussed in the introduction, a risk measure on a Brownian filtration is not always the unique solution of a backward SDE (respectively a PDE), unless an additional so-called domination condition is satisfied.
In general, the risk measure can only be proved to be the minimal supersolution of a backward SDE (resp.\ PDE), see \citet[Theorem 3.2]{Del-Peng-Ros} and \citet[Theorem 4.7]{tarpodual}. 
}

In our typically time-insconsistent setting, it is clear from the proof of Theorem \ref{thm:exist} that our PDE \eqref{eq:hjb}, accompanied with relevant boundary conditions in the bounded-domain case (see Remark \ref{rmk final uniqueness} for more on this), admits a unique solution if it satisfies a \emph{comparison principle} in the following sense: If $v^1$ and $v^2$ are respectively an upper semicontinuous viscosity subsolution and a lower semicontinuous viscosity supersolution, both of them with polynomial growth, and $v^1(T,\cdot)\le v^2(T,\cdot)$, then $v^1\le v^2$ holds everywhere.
However, little is known about comparison for PDEs in the generality of \eqref{eq:hjb}, owing to the character of the problem as we have repeatedly mentioned. 
By a formal optimization, the PDE \eqref{eq:hjb} can be rewritten has 
\begin{equation*}
	\partial_tV + b(s,y)\partial_yV + \frac{1}{2}tr(\sigma(s,y)\sigma(s,y)')\partial^2_{yy}V + zg(s,y) = \frac{1}{2}\sigma^2(s,y)\frac{(\partial_{ys}^2V)^2}{\partial^2_{zz}V}.
\end{equation*}
Fully nonlinear parabolic PDEs of this kind were studied e.g by \citet{CSTV2007} using the notion of ``BSDE with gamma constraints,'' and uniqueness was obtained there by assuming that a comparison principle holds. It is an open question to establish the actual validity of a comparison principle in our framework.
\end{remark}

\subsection{Proof of Theorem \ref{thm:exist}} \label{subsec proof thm}
The proof of Theorem \ref{thm:exist} will be split into several auxiliary results. The proofs of Lemma \ref{thm:rep} and Proposition \ref{lem PDE} are left to the appendix. 
\begin{lemma}
\label{thm:rep}
The functional $\rho$ maps $L^\infty$ to $\mathbb{R}$, and is a convex, increasing and cash-invariant functional satisfying the representation
\begin{equation}
	\rho(X)\,=\, 	\sup\limits_{Z\in{\cal Z} }\left(E[XZ] -E[l^*( Z)]  \right),\quad X \in L^\infty,\label{dual_rep1}
\end{equation}
where ${\cal Z}:=\{Z \in L^{1}_+: Z\in \text{dom}(l^*),\,E[Z]=1\text{ and } Z\ge c\text{ for some } c>0\}$.
\end{lemma}

Let ${\cal L}_{b}$ be the set of $\mathbb{R}^d$-valued progressively measurable processes that are essentially bounded
and for each $\beta \in {\cal L}_{b}$, put $$\textstyle Z^\beta_s:=\exp\left( \int_0^s\beta_t\,dW_t - \frac{1}{2}\int_0^s|\beta_t|^2\,dt\right),\,\, s \in [0,T].$$ 
In our Brownian filtration, the OCE can be represented in terms of processes $\beta \in {\cal L}_b$ as we now show.%

\begin{proposition}\label{lem PDE}
For every $X \in L^{\infty}$ we have
\begin{equation}\label{eq:reduction}
\rho(X)\,\,=\,\, \sup_{\beta \in \mathcal{L}_b} E\left [ XZ^\beta_T - l^*\left( Z^\beta_T \right )  \right ],
\end{equation}
i.e.\ $\rho(X)$ can be computed over densities with essentially bounded stochastic logarithms.
\end{proposition}

From Proposition \ref{lem PDE} the computation of $\rho(X)$ can be reduced to solving a stochastic optimal control problem of degenerate form.
This connection is made precise in Proposition \ref{prop:struct-V}.(a), in which $Z^{s,z,\beta}$ denotes the solution of the controlled stochastic differential equation (SDE)
	\begin{equation}
	\label{eq:sde_Z}\textstyle
		Z_t = z + \int_s^t\beta_uZ_u\,dW_u\quad t\ge s, \quad \text{for  } \beta \in {\cal L}_b.
	\end{equation}
Observe that the stochastic control problem therein is set in an enlarged state space (of $Y$'s and $Z$'s, where the former is actually uncontrolled). The usefulness of reducing the optimization problem in Proposition \ref{lem PDE} to $\beta \in \mathcal{L}_b$ is that it will allow us to approximate the forthcoming stochastic control problem by simpler ones (namely with compact control constraints), for which a stronger theory of viscosity solutions is available. 

\begin{proposition}
\label{prop:struct-V}
	Assume that (A1) holds and let $f:\mathbb{R}^m\to \mathbb{R}$ and $g:[0,T]\times\mathbb{R}^m\to \mathbb{R}$ be two bounded measurable functions and put $X:= f(Y_T^{0,y}) + \int_0^Tg(t,Y_t^{0,y})\,dt$, $y\in\mathbb{R}^m$.
	Then: 
	\begin{itemize}
	\item[(a)] { $\rho(X)$ is the value of a stochastic optimal control problem with state processes $(Y,Z)$:
	\begin{equation}\textstyle\label{SOC}\rho(X)\,\,=\,\,\sup\limits_{\beta \in \mathcal{L}_b} E\left[ f(Y^{0,y}_T)Z_T^{0,1,\beta} - {l}^*\left ( Z_T^{0,1,\beta}\right ) +\int_0^T g(t,Y^{0,y}_t)Z_t^{0,1,\beta}dt \right] .
\end{equation} 
	In particular $\rho(X) = V(0,y,1)$ where $V$ is the value function of \eqref{SOC}, namely
	\begin{equation}\textstyle
	\label{eq:controlproblem}
 		V(s,y,z):=\sup\limits_{\beta \in \mathcal{L}_b} E\left[ f(Y^{s,y}_T)Z_T^{s,z,\beta} - {l}^*\left ( Z_T^{s,z,\beta}\right ) +\int_s^T g(t,Y^{s,y}_t)Z_t^{s,z,\beta}dt \right],
 	\end{equation} 
 	 for all $(s,y,z)\in [0,T]\times \mathbb{R}^m\times {\cal O}$. }
	\item[(b)] $V$ is concave in $z$ and satisfies the equivalent representations
	\begin{align}\label{eq: rep dyn 1}
 	 V(s,y,z)&\textstyle = \inf\limits_{r\in\mathbb{R}}\left\{ E\left [l\left (  f(Y^{s,y}_T) +\int_s^T g(t,Y^{s,y}_t)dt-r \right )\right ]+rz  \right \} \\
 	 &\textstyle = \rho^{l_z}\left( z \left [f(Y^{s,y}_T)  +\int_s^T g(t,Y^{s,y}_t)dt\right ]\right ),\label{eq: rep dyn 2}
\end{align}	
where $\rho^{l_z}$ is the OCE corresponding to the (not normalized) loss function $ l_z(x) := l(x/z)$.
	\item[(c)] $V$ is continuous on $[0,T]\times\mathbb{R}^m\times {\cal O}$.	
	\end{itemize}

\end{proposition}
\begin{proof}$ $\\
$(a)$ The identity \eqref{SOC} follows from Proposition \ref{lem PDE} and It\^o's formula. We think of \eqref{SOC} as an optimal control of the diffusion processes $Y$ and $Z$. The function $V$ in \eqref{eq:controlproblem} is then naturally its value function.\\
$(b)$
From \eqref{eq:controlproblem}, upon writing $Z^{s,z,\beta}=zZ^{s,1,\beta}$ and using the definition of OCEs (notice that $l_z^*(x)=l^*(zx)$), we obtain the representations \eqref{eq: rep dyn 1}-\eqref{eq: rep dyn 2}. From \eqref{eq: rep dyn 1}, $V$ is clearly concave in $z$.
\\
$(c)$ In order to prove continuity of $V$, notice that the infimum in \eqref{eq: rep dyn 1} can be restricted to a compact interval.
In fact, let $\|X\|_{\infty}\leq C$, so clearly $\rho^{l_z}(X)\leq Cz-l^*(z)\leq Cz$. For a fixed $z\in {\cal O}$ we get that any $1$-optimizer $r$ for $\rho^{l_z}(X)$ must satisfy 
$$
l(-C-r) +rz\leq E[l(X-r)]+rz \leq \rho^{l_z}(X) +1\leq Cz +1 ,
$$
and so for any $p\in \mbox{dom}(l^*) $ we have
$$-l^*(p)-Cp+r(z-p)\leq Cz+1.$$
Choosing either $p>z$ or $p<z$ shows that $r$ must a priori lie in a compact interval which only depends on $C$ and $z$. 

Now we prove the continuity claim. Take $(s_n,y_n,z_n)$ converging to $(s,y,z)$, all of them in $[0,T]\times\mathbb{R}^m\times {\cal O}$. Since $f,g$ are bounded and $z_n$ is converging in ${\cal O}$, the previous argument shows that the infima in \eqref{eq: rep dyn 1} for $V(s_n,y_n,z_n)$ can be computed for $r$ in a compact interval independent of $n$. From this we get $\liminf V(s_n,y_n,z_n)\geq V(s,y,z)$, by the a.s.\ continuity of $(s,y)\mapsto Y_T^{s,y}$, dominated convergence and the continuity of $l$, which together imply that $(s,y,z,r)\mapsto r+ E\left[l\left(f(Y^{s,y}_T)+\int_s^T g(t,Y_t^{s,y})dt-r/z\right)\right ]$ is continuous. But from \eqref{eq: rep dyn 1}, we also get that $V$ is upper semicontinuous, as an infimum of continuous functions. This finishes the proof.
\end{proof}
{
\begin{remark}
\label{rem growth}
By \eqref{eq: rep dyn 1} we get $V(s,y,z)\leq E[l(f(Y_T^{s,y}) +\int_s^T g(t,Y_t^{s,y})dt) ]$, so under the boundedness assumptions on $f,g$ and continuity of $l$ we see that $V$ is bounded from above. We also obtain $V(s,y,z)\geq zC-l^*(z)$, for instance by \eqref{eq:controlproblem}. Thus the growth of $V$ is only interesting in the $z$-component, and is fully captured by $l^*$. Hence observe that if $l^*$ is finite, then $l^*$ (equivalently $V$) has polynomial growth if and only if $l$ grows at least polynomially. On the other hand, if ${dom(l^*)}$ is bounded, then $V$ has polynomial growth immediately.  
Thus the assumptions for minimality in Theorem \ref{thm:exist} apply to $V$ under the given conditions.

\end{remark}
}
\begin{remark}
\label{rem:roughness of V}
Taking the process $Y$ to be constant (i.e.\ $b,\sigma\equiv 0$), we find $V(s,y,z)=\big(f(y)+\int_s^Tg(t,y)dt\big)z-l^*(z)$. Of course, if either of $l^*$, $f$ or $g$ is rough, then so will be $V$. This extends the phenomenon in Example \ref{ex preliminary} to the value function as a whole. On the other hand, if $l^*$ is twice differentiable, then it is elementary to show that this $V$ is a classical solution to the corresponding \eqref{eq:hjb}, despite the potential roughness in $y$, since there are no $y$-derivatives involved there.
\end{remark}

As proven by \citet{kupper02}, most law-invariant risk measures (of which OCE form a subfamily) are not time consistent in the sense described in the introduction. 
Equations \eqref{eq: rep dyn 1}-\eqref{eq: rep dyn 2} above 
can be seen as substitutes for time-consistency. They rely on the idea of enlarging the state space. 
This idea is further developed in Proposition \ref{prop: Bellman gnral} below,
and culminates in the dynamic programming principle of Corollary \ref{cor:bellman} thereafter.
This is how we induce time-consistency into the problem.
Note however that one needs to keep track of the state $Z$, and that both the claim and the loss function in  \eqref{eq: rep dyn 2} need to be scaled properly. This is most apparent for conditional value-at-risk, as first noted by \citet{PflugPichlerinconsistency1} in discrete-time; see Section \ref{sec CVAR} below for the explicit expression in the present continuous-time setting. We would like to stress that it is the present stochastic control perspective, based on the dual representation of OCEs which permits to unearth the pleasant dynamic properties we have referred to (namely a dynamic programming principle and eventually a PDE characterization); this seems to be a strong advantage of the method as opposed to a purely primal perspective.

Let $X\in L^\infty(\mathcal F_T)$ and define 
\begin{equation}\label{def OCE cond}
(s,\eta)\in [0,T]\times L^0(\mathcal F_s) \mapsto v(X,s,\eta): = \essinf_{r\in \mathbb R}\left\{ E[l(X-r)|\mathcal F_s] +r\eta \right \},
\end{equation}
and observe that $\rho(X)=v(X,0,1)$. We have the following Bellman-type principle which actually holds in general filtrations. {We stress that this is more or less easily obtained because of the (primal) structure of OCEs; indeed, since \eqref{def OCE cond} is just a scalar minimization problem, there is no need for deep measurable selection arguments (as opposed to e.g.\ the situation in stochastic control theory).}

\begin{proposition}\label{prop: Bellman gnral}
For  $0\leq s\leq t\leq T$ and $\eta\in L^0(\mathcal F_s)\cap \text{dom}(l^*) $ we have
\begin{align}\textstyle
	v(X,s,\eta)&=\, 	\esssup\limits_{\substack{Z \in L^1_+({\cal F}_t)\\ \eta Z\in \text{dom}(l^*),E[ Z|\mathcal F_s]=1}} E\left[  \essinf_{r\in\mathbb{R}}\bigl\{ E[l (X-r)|{\cal F}_t] +r\eta Z \bigr \}\,\,| \mathcal F_s \right]\label{dual_rep_recursive gnral} \\
	 &=\, \esssup\limits_{\substack{Z \in L^1_+({\cal F}_t)\\ \eta Z\in \text{dom}(l^*),E[ Z|\mathcal F_s]=1}} E\left[ \,\, v(X,t,\eta Z)\,\,| \mathcal F_s \right]  \notag .
\end{align}
In particular, 
\begin{equation}\textstyle
	\rho(X)\,=\, 	\sup\limits_{\substack{Z \in L^1_+({\cal F}_t)\\ Z\in \text{dom}(l^*),E[ Z]=1}} E\left[  \essinf_{r\in\mathbb{R}}\bigl( E[l (X-r)|{\cal F}_t] +rZ \bigr ) \right].\label{dual_rep_recursive}
\end{equation}
\end{proposition}

\begin{proof}
It is elementary that the r.h.s.\ of \eqref{dual_rep_recursive gnral} is almost surely bounded from above by
$$\textstyle \essinf_{r\in\mathbb R}\esssup\limits_{\substack{Z \in L^1_+({\cal F}_t)\\ \eta Z\in \text{dom}(l^*),E[ Z|\mathcal F_s]=1}}  E\left[  \,E[l (X-r)|{\cal F}_t] +r\eta Z \,\,| \mathcal F_s \right],$$
which is equal to $v(X,s,\eta)$ by the tower property, the measurability of $\eta$ and the fact that $E[ Z|\mathcal F_s]=1$. So we only need to establish the opposite inequality. 

We first notice that by definition of $l^*$ and a measurable selection argument, $v(X,s,\eta)$ has the convex dual representation
\begin{equation}
	v(X,s,\eta)\,=\, 	\esssup\limits_{Z\in{\cal Z}_s }E\left[X\eta Z-l^*( \eta Z)\, |\mathcal F_s  \right ],\quad X \in L^\infty(\mathcal F_T),\label{dual_rep conditional}
\end{equation}
where $\mathcal Z_s=\{Z\in L^1_+(\mathcal F_T): E[Z|\mathcal F_s]=1\}$. This is just the robust representation of a conditional risk measure which, while being similar to a conditional OCE, is only translation invariant by a factor of $\eta$. By arguing as in the proof of Lemma \ref{thm:rep} we may assume that, for the $Z$ over which the supremum in \eqref{dual_rep conditional} is computed, it holds that the essential range\footnote{The essential range of $\eta Z$ is $range(\eta Z):=[\essinf \eta Z,\esssup \eta Z]$.} of $\eta Z$ is contained in $\text{int(dom}(l^*)\text{)}$.  Let $\{R^{\tilde{Z}}\}_{\tilde{Z}\in L^1_{+}({\cal F}_t)}\subset L^{\infty}({\cal F}_t)$. 
From \eqref{dual_rep conditional}, the observation made and Fenchel-Young's inequality we see
\begin{align}
v(X,s,\eta)& \textstyle =  	\esssup\limits_{ \substack{Z\in L^1_+({\cal F}_T),E[ Z|\mathcal F_s]=1\\ range(\eta Z)\subset \text{int(dom}(l^*)\text{)}}}E\left[(X-R^{E[Z|{\cal F}_t]})\eta Z -l^*(\eta Z) + R^{E[Z|{\cal F}_t]}\eta Z \, |\mathcal F_s  \right]\notag\\
&\textstyle \leq 	\esssup\limits_{\substack{Z\in  L^1_+({\cal F}_T),E[ Z|\mathcal F_s]=1\\range(\eta Z)\subset \text{int(dom}(l^*)\text{)}\\ }} E\left[ E[l(X-R^{E[Z|{\cal F}_t]})|{\cal F}_t] + R^{E[Z|{\cal F}_t]}\eta {E[Z|{\cal F}_t]}  \, |\mathcal F_s \right]\notag \\
&\textstyle \leq 	\esssup\limits_{\substack{Z\in  L^1_+({\cal F}_t),E[ Z|\mathcal F_s]=1\\range(\eta Z)\subset \text{int(dom}(l^*)\text{)}\\ }} E\left[ E[l(X-R^Z)|{\cal F}_t] + R^Z\eta Z  \, |\mathcal F_s \right].\label{eq upper essinf gnral}
\end{align}
The result will follow by choosing (for each $Z\in L^{1}_{+}({\cal F}_t)$ with $E[ Z|\mathcal F_s]=1$ and $range(\eta Z)\subset \text{int(dom}(l^*)\text{)}$) the functions $R^Z$ wisely. Let us define 
$$\textstyle I(Z)=\essinf_{R \in  L^{\infty}({\cal F}_t)}\{E[l(X-R)|{\cal F}_t]+ R\eta Z\} > - \infty,$$
where the inequality follows from $I(Z)\geq \inf_r\{l(-\|X\|_{\infty}-r)+r\eta Z\}=-\eta Z \|X\|_{\infty} -l^*(\eta Z)$. Observe that the family $\{E[l(X-R)\mid{\cal F}_t ]+ R\eta Z\,:\,\, R \in L^\infty({\cal F}_t) \}$ is directed downwards. Thus there is a feasible sequence $\{R_n\}$ such that $I(Z)$ is the decreasing limit of $E[l(X-R_n)|{\cal F}_t]+ R_n\eta Z$; this follows from \citet[Appendix A.5]{FS3dr}.
Furthermore, arguing as{} in the proof of Proposition \ref{prop:struct-V}.(c) we may assume w.l{}.o.g.\ that $\{R_n\}$ is uniformly essentially bounded. It is then elementary to construct from this, for any $\varepsilon>0$, an ${R}_{\varepsilon}=R_\varepsilon(Z)\in L^{\infty}({\cal F}_t) $ such that $I(Z)\geq -\varepsilon + E[l(X-R_{\varepsilon})|{\cal F}_t]+ R_{\varepsilon}\eta Z$. Taking $R^Z=R_{\varepsilon}$ in \eqref{eq upper essinf gnral} gives
$$\textstyle v(X,s,\eta)-\varepsilon\leq \esssup\limits_{\substack{Z\in  L^1_+({\cal F}_t)\\ \eta Z\in \text{dom}(l^*) \\ E[ Z|\mathcal F_s]=1}} E[I(Z)|\mathcal F_s]\leq \esssup\limits_{\substack{Z\in  L^1_+({\cal F}_t)\\ \eta Z\in \text{dom}(l^*) \\ E[ Z|\mathcal F_s]=1} }E\left[ \essinf_{r\in\mathbb{R}}\left ( E[l(X-r)|{\cal F}_t]+ r\eta Z  \right )  |\mathcal F_s\right], $$
as $\mathbb{R}\subset L^\infty({\cal F}_t)$. We conclude taking $\varepsilon\to 0$.
\end{proof}

The previous Bellman-type principle (for $v$ as in \eqref{dual_rep conditional}) becomes more familiar in the Markovian setting (for $V$) we have mostly discussed so far, to wit:

\begin{corollary}
\label{cor:bellman}
	Assume that (A1)-(A2) hold.
	Then Bellman's dynamic programming principle is satisfied, that is, for every $0\le s\le T$ and $\theta$ a stopping time with values in $[s,T]$ we have
	\begin{equation}\textstyle
	\label{eq:Bellman}
 		V(s,y,z)=\sup\limits_{\substack{\beta \in \mathcal{L}_b}} E\left[  \int_s^{\theta} g(t,Y^{s,y}_t)Z_t^{s,z,\beta}dt + V\left(\theta, Y_{\theta}^{s,y},Z_{\theta}^{s,z,\beta} \right ) \right], \,\, \mbox{ }y\in {\mathbb R^m},z\in \text{dom}(l^*).
	\end{equation}
\end{corollary}

\begin{proof}
That the l.h.s.\ is smaller than the r.h.s.\ is a classical application of the flow property for the strong (and unique) solution of the system for $(Y,Z)$ when $(s,y,z,\beta)$ are specified. See for example \citet[Chap. 4, Theorem 3.3]{YongZhou} and its proof. For the converse inequality, we shall establish
\begin{align}
\label{eq ineq useful DPP}
V(s,y,z)\geq \textstyle \sup\limits_{\substack{Z \in L^1_+({\cal F}_\theta)\\ zZ\in \text{dom}(l^*),E[ Z|\mathcal F_s]=1}} E\left [ z Z\int_s^{\theta} g(t,Y^{s,y}_t) dt + V\left(\theta, Y_{\theta}^{s,y},zZ \right ) \right ].
\end{align}
This and arguments as in the proof of Proposition \ref{lem PDE} (permitting to reduce to the case of essentially bounded $\beta$ after representing $Z$ in the Brownian filtration) yield the desired result. We start observing, by \eqref{eq: rep dyn 1} and the tower property, that%
\begin{align*}
V(s,y,z)&=  \textstyle\inf_r E\left [\,\, E\left[l\left( f(Y_T^{s,y})+\int_s^T g(t,Y_t^{s,y})dt -r\right)\Bigl | \,\mathcal{F}_s \right]+rz \,\, \right ]\\
&\geq \textstyle E\left [\,\,\essinf_r\left \{ E\left[l\left( f(Y_T^{s,y})+\int_s^T g(t,Y_t^{s,y})dt -r\right)\Bigl | \,\mathcal{F}_s \right]+rz \right\}\,\, \right ].
\end{align*}
Using Proposition \ref{prop: Bellman gnral} applied 
to the claim $\tilde{X}:= f(Y_T^{s,y}) +\int_s^T g(t,Y_t^{s,y})dt$, we get:
\begin{align*}
&V(s,y,z)\geq \textstyle  E\left[\,\, \esssup\limits_{\substack{Z \in L^1_+({\cal F}_\theta)\\ z Z\in \text{dom}(l^*)\\E[ Z|\mathcal F_s]=1}}  E\left [\essinf_r\left \{ E\left[l\left( f(Y_T^{s,y})+\int_s^T g(t,Y_t^{s,y})dt -r\right)\Bigl | \,\mathcal{F}_\theta \right]+rzZ \right\} \, \Bigl |\, \mathcal{F}_s \right ]\,\,\right] \\
=& \textstyle 
 E\left[\,\,  \esssup\limits_{\substack{Z \in L^1_+({\cal F}_\theta)\\ z Z\in \text{dom}(l^*)\\E[ Z|\mathcal F_s]=1}}  E\left [\essinf_r\left \{ E\left[l\left( f(Y_T^{s,y})+\int_\theta^T g(t,Y_t^{s,y})dt -r\right)\Bigl | \,\mathcal{F}_\theta \right]+rzZ \right\} +zZ\int_s^\theta g(t,Y^{s,y}_t)dt \, \Bigl |\, \mathcal{F}_s \right ]\,\,\right],
\end{align*}
by separating the integral and changing variables $r -\int_s^\theta g(t,Y_t)dt\to r$,  which is allowed thanks to the $\mathcal F_\theta$-conditional expectation. We can then further bound from below and use the tower property:
\begin{align*}
&V(s,y,z)\geq  \\ & \textstyle 
 \sup\limits_{\substack{Z \in L^1_+({\cal F}_\theta)\\ z Z\in \text{dom}(l^*)\\E[ Z|\mathcal F_s]=1}}E\left[\,\,    E\left [\essinf_r\left \{ E\left[l\left( f(Y_T^{s,y})+\int_\theta^T g(t,Y_t^{s,y})dt -r\right)\Bigl | \,\mathcal{F}_\theta \right]+rzZ \right\} +zZ\int_s^\theta g(t,Y^{s,y}_t)dt \, \Bigl |\, \mathcal{F}_s \right ]\,\,\right]\\
 =&
 \textstyle 
 \sup\limits_{\substack{Z \in L^1_+({\cal F}_\theta)\\ z Z\in \text{dom}(l^*)\\E[ Z|\mathcal F_s]=1}}E\left[\,\,    \essinf_r\left \{ E\left[l\left( f(Y_T^{s,y})+\int_\theta^T g(t,Y_t^{s,y})dt -r\right)\Bigl | \,\mathcal{F}_\theta \right]+rzZ \right\} +zZ\int_s^\theta g(t,Y^{s,y}_t)dt \,\,\right].
\end{align*}
Observe that by flow property arguments as in  \citet[Chap. 4, Lemma 3.2]{YongZhou}, we have
\begin{align*}  & \textstyle \essinf_r\left \{ E\left[l\left( f(Y_T^{s,y})+\int_\theta^T g(t,Y_t^{s,y})dt -r\right)\Bigl | \,\mathcal{F}_\theta \right]+rzZ \right\}  \\=& \textstyle \inf_r\left \{ E\left[l\left( f(Y_T^{s,y})+\int_\theta^T g(t,Y_t^{s,y})dt -r\right)\Bigl | \,\mathcal{F}_\theta \right]+rzZ \right\} \\
=& \textstyle 
\inf_r\left \{ E\left[l\left( f(Y_T^{\theta,Y_\theta^{s,y}})+\int_\theta^T g(t,Y_t^{\theta,Y_\theta^{s,y}})dt -r\right)\right]+rzZ \right\} \\
=&\textstyle V(\theta,Y_\theta^{s,y} ,zZ),
\end{align*}
where the first equality comes from the fact that its r.h.s.\ is measurable (it suffices to take infimum over the rational numbers), and the equality follows by \eqref{eq: rep dyn 1}. This identity and the previous inequality prove \eqref{eq ineq useful DPP}.
\end{proof}

We conclude this section with the proof of Theorem \ref{thm:exist}. {In light of Proposition \ref{prop:struct-V}, this boils down to proving that the value function $V$ there is the (minimal) viscosity solution of the HJB equation \eqref{eq:hjb}.}

 \begin{proof}(of Theorem \ref{thm:exist})$ $\\
 	STEP 1: \emph{Viscosity subsolution property of $V$.}\\
Let $n\in \mathbb{N}\setminus\{0\}$, put ${\cal L}_b^n:=\{\beta \in {\cal L}_b:  |\beta|\le n\}$ and consider the control problem
	\begin{equation}\textstyle
	\label{eq:valuetruncated}\textstyle
 		V^{n}(s,y,z):=\sup\limits_{\substack{\beta \in \mathcal{L}_b^n}} E\left[ f(Y^{s,y}_T)Z_T^{s,z,\beta} - l^*\left ( Z_T^{s,z,\beta}\right ) +\int_s^T g(t,Y^{s,y}_t)Z_t^{s,z,\beta}dt \right]
 	\end{equation} 
 	with $(s,y,z)\in [0,T]\times \mathbb{R}^m\times {\cal O}$.
It is associated to the HJB equation
\begin{align}
\label{eq:hjbtruncate}
\begin{cases}
	&\partial_tV + H^n(t, y,z, DV, D^2V) = 0,\quad (t,y,z)\in [0,T)\times \mathbb{R}^m\times {\cal O}\\
	&V(T,y,z) = \psi(y,z),\quad (y,z)\in \mathbb{R}^m\times {\cal O}
\end{cases}
\end{align}
where 
	\begin{equation*}
		H^n(t, x, \gamma, \Gamma) := \langle B(t,y,z), \gamma\rangle + z g(t, y) + \frac{1}{2}\sup_{ |\beta|\le n}tr(AA'\Gamma).
	\end{equation*}
An application of the flow property for the
strong (and unique) solution of the system of SDEs for $(Y, Z)$ when $(s,y,z,\beta)$ are fixed, shows that $V^n$ satisfies
\begin{equation}\textstyle
\label{eq:half_DPP_n}
	V^n(s,y,z) \le \sup\limits_{\beta \in {\cal L}^n_b}E\left[\int_s^\theta g(t,Y^{s,y}_t)Z^{s,z,\beta}_t\,dt + V^n\left(\theta, Y^{s,y}_\theta, Z^{s,z,\beta}_\theta\right) \right]
\end{equation}
for all $(s,y,z)\in [0,T]\times \mathbb{R}^m\times {\cal O}$ and $[s,T]$-valued stopping times $\theta$. 
 With this inequality at hand, it is now routine to show that $V^n$ is a viscosity subsolution of \eqref{eq:hjbtruncate}.
We provide the argument for the sake of completeness; the method we use is similar in spirit to \citet[Section 5]{Neu-Nutz17}. It is well-known (see e.g.\ \citet[Chap. 2, Sect. 2.6, Theorem. 6.1]{Flem-Soner-second}) that when testing the subsolution property, we may assume that the test function $\varphi $ is smooth with bounded derivatives and that 
 $V^n -\varphi$ has a global maximum at $x=(s,y,z)\in [0,T]\times \mathbb{R}^m\times {\cal O}$ with $V^n(x) = \varphi(x)$.
If $s = T$, then $\varphi(x) = \psi(x)$.
Assuming $s<T$, then by \eqref{eq:half_DPP_n} we have
\begin{equation*}\textstyle
	0\le \sup\limits_{\beta\in {\cal L}^n_b}E\left[\int_s^{s+u}g(t,Y^{s,y}_t)Z^{s,z,\beta}_t\,dt+\varphi(s+u, Y_{s+u}^{s,y}, Z^{s,z,\beta}_{s+u}) - \varphi(s,y,z) \right]
\end{equation*}
for all $u\in (0,T-s)$.
Applying It\^o's formula to $t\mapsto \varphi(t, Y_{t}^{s,y}, Z^{s,z,\beta}_{t})$ yields
\begin{align*}
	0\,\le &\textstyle \sup_{\beta \in {\cal L}^n_b}\int_s^{s+u}E\Bigl [ g(t,Y^{s,y}_t)Z^{s,z,\beta}_t+  \partial_y\varphi(t, Y^{s,y}_t, Z^{s,z,\beta}_t)b(t,Y^{s,y}_t) + \partial_t\varphi(t, Y_t^{s,y}, Z^{s,z,\beta}_t)\\
	& \textstyle + \frac{1}{2}tr(\partial_{yy}\varphi(t, Y^{s,y}_t, Z^{s,z,\beta}_t)\sigma\sigma'(t,Y^{s,y}_t))+\frac{1}{2}\partial_{zz}\varphi(t, Y^{s,y}_t, Z^{s,z,\beta}_t)|\beta_t|^2(Z_t^{s,z,\beta})^2\\ &\textstyle +\partial_{yz}\varphi(t, Y^{s,y}_t, Z^{s,z,\beta}_t)\sigma(t,Y^{s,y}_t)\beta_tZ^{s,z,\beta}_t \Bigr ]\,dt .
\end{align*}
Since $b,\sigma$ and $\varphi$ (as well as its derivatives) are Lipschitz continuous, and by Cauchy-Schwarz inequality and classical SDE estimates, there is a continuous function $t\mapsto R(t)$ with $R(0)=0$, further parametrized only by $b,\sigma,s,\varphi,n,z,y$, such that
\begin{align*}
&0\, \le	 \textstyle \sup\limits_{\beta \in {\cal L}^n_b}\int_s^{s+u} R(t-s)+ E\left[g(t,Y^{s,y}_t)Z^{s,z,\beta}_t+  \partial_y\varphi(s,y, z)b(t,Y^{s,y}_t) + \partial_t\varphi(s,y, z)\right ] \\
	&\textstyle + E\left[ \partial_{yz}\varphi(s,y, z)\sigma(t,Y^{s,y}_t)\beta_tZ^{s,z,\beta}_t + \frac{1}{2}\left(tr(\partial_{yy}\varphi(s,y, z)\sigma\sigma'(t,Y^{s,y}_t))+\partial_{zz}\varphi(s,y, z)|\beta_t|^2(Z_t^{s,z,\beta})^2\right)\right ]dt.
\end{align*}
Observe that having a uniform bound on $\beta $ was essential here. As a consequence, we have 
\begin{align*}
	 &0\le \textstyle\int_s^{s+u}R(t-s) + E\left[  \partial_y\varphi(s,y, z)b(t,Y^{s,y}_t) + \partial_t\varphi(s,y, z)+\frac{1}{2}tr(\partial_{yy}\varphi(s,y, z)\sigma\sigma'(t,Y^{s,y}_t))\right]\\
	 &\textstyle+E\left[g(t,Y^{s,y}_t)Z_t^{s,z,\beta}+\sup_{\beta \in \mathbb{R}^d:|\beta|\le n}\frac{1}{2}\partial_{zz}\varphi(s,y, z)|\beta|^2(Z_t^{s,z,\beta})^2+\partial_{yz}\varphi(s,y, z)\sigma(t,Y^{s,y}_t)\beta Z_t^{s,z,\beta}\right]dt.
\end{align*}
Dividing by $u$, using dominated convergence, and letting $u$ go to $0$ gives
\begin{equation*}
	\partial_t \varphi(s,y,z) + H^n(s,y,z, D\varphi(s,y,z), D^2\varphi(s,y,z))\ge0,
\end{equation*}
showing that $V^n$ is a viscosity subsolution of \eqref{eq:hjbtruncate}.

We now adapt a usual stability argument to our setting in order to show that $V$ is a viscosity subsolution of \eqref{eq:hjb} in the sense of Definition \ref{def viscosity}.
It is easy to see that $V^n$ is jointly continuous, and by Proposition \ref{prop:struct-V}.(c) we know that $V$ is continuous. Crucially, we have that $V^n$ increases to $V$; see \eqref{eq:controlproblem}. Combining these facts with Dini's lemma shows that $(V^{n})$ converges to $V$ uniformly on compacts.
	Let us show that $V$ is then a viscosity subsolution of \eqref{eq:hjb}. 	Let $\varphi \in C^2$ be a test function such that $V -\varphi$ has a strict local maximum at $x_0=(s_0,y_0,z_0)\in [0,T)\times \mathbb{R}^m\times {\cal O}$ and $(x_0, D\varphi(x_0), D^2\varphi(x_0))\in \text{int dom(H)}$. It is routine that the case of non-strict local maximum can be obtained as a consequence of the strict-case. Let $B_{r}(x_0):=\{x: |x-x_0|\le r\}$, with $r$ small enough so $x_0$ is the maximum of $V -\varphi$ on $B_{r}(x_0)$.	Denote by $x^n$ the point at which $V^{n}-\varphi$ reaches its maximum in $B_r(x_0)$. We may suppose $x_n\to \bar{x}$. The uniform convergence on $B_r(x_0)$ of $V^n$ to $V$ yields $(V-\varphi)(x)= \lim (V^n-\varphi)(x) \leq \lim(V^n-\varphi)(x_n)=(V-\varphi)(\bar{x})$, and we conclude $\bar{x}=x_0$.
		As $V^{n}$ is a viscosity subsolution of \eqref{eq:hjbtruncate}, we have by definition
	\begin{equation}
		\label{eq:visc_subVn}
		\partial_t\varphi(x^n) + H^n( x^n, D\varphi(x^n), D^2\varphi(x^n)) \ge 0\quad \text{for all }n\in \mathbb{N}.
	\end{equation}
 	The sequence $(H^n)$ increases pointwise to $H$, so that taking the limit in \eqref{eq:visc_subVn} yields
	\begin{equation}
	\label{eq:visc_sub}
		\partial_t{\varphi}(x_0) + H( x_0, D\varphi(x_0), D^2\varphi(x_0))\ge 0.
	\end{equation}
	Indeed, by \eqref{eq:visc_subVn} it holds $\partial_t\varphi(x^n) + H( x^n, D\varphi(x^n), D^2\varphi(x^n)) \ge 0$ for all $n$, and
by assumption $(x_0, D\varphi(x_0), D^2\varphi(x_0))\in \text{int dom(H)}$, so for $n$ large $(x^n, D\varphi(x^n), D^2\varphi(x^n))\in \text{int dom(H)}$. Thus \eqref{eq:visc_sub} follows, since clearly $H$ is continuous in the interior of its domain.

 	STEP 2: \emph{Viscosity supersolution property of $V$.}
That $V$ is a viscosity supersolution of \eqref{eq:hjb} follows from the crucial DPP given in Corollary \ref{cor:bellman} and classical arguments.
Note that in this case the truncation is not necessary since it is enough to argue with constant controls.

	STEP 3: \emph{Minimality of $V$.}
	Let $w$ be a viscosity supersolution of \eqref{eq:hjb} with polynomial growth.
	Since $H^n\le H$, it follows that $w$ is also a viscosity supersolution of \eqref{eq:hjbtruncate}, for every $n$.
	The function $V^n$ is a viscosity subsolution of \eqref{eq:hjbtruncate}.
	Notice that since $H^n$ is finite our definition of viscosity solution for \eqref{eq:hjbtruncate} coincides with the usual definition.
	Therefore we may apply the comparison theorem for unbounded domains under the polynomial growth assumption, as in \citet{Touzibook} or \citet{Pham}, obtaining $ V^{n}\le w$. 
	Passing to the limit implies $ V\le w$.
 \end{proof}

We close this section with a remark on the relationship between the primal and dual representation of our risk measures, given the knowledge of value function $V$.

\begin{remark}
	It can be tempting to use the primal representation \eqref{eq inf min} of $\rho(f(T_T))$ to derive a dynamic representation, since for each $r \in \mathbb{R}$ fixed, $E[l(f(Y_T)-r)] +r$ is the initial value of the (viscosity) solution of a linear PDE.
	But the ``optimal cash-allocation,'' namely the number $r^*$ such that $\rho(f(Y_T)) = E[l(f(Y_T)-r^*)] +r^*$, is not known explicitly. Thus such linear PDE does not provide a meaningful representation for $\rho(f(Y_T))$.
	On the other hand, if $l$ is differentiable, then  by e.g \cite{Dra-Kup-Pap} the optimal $Z^*$ in \eqref{dual_rep1} is given by $Z^*= l'(f(Y_T)-r^*)=: Z^{\beta^*}_T$.
	If $V$ is a classical solution (as in Proposition \ref{prop mmv} below) then the process $\beta^*$ can be obtained by verification arguments, and this allows in turn to compute $r^*$. Actually if \eqref{eq: rep dyn 1} is attained uniquely at $r^*(s,y,z)$, then by the compactness obtained in part (c) of the proof of Proposition \ref{prop:struct-V} and the envelope theorem we would have: $$V(s,y,z) \text{ is differentiable in $z$, and }\,\,r^*(s,y,z)= \partial_z V(s,y,z).$$
	More generally, it is expected that the set of optimal $r$'s coincides with the partial $z$-superdifferential of $V$. In this way, the computation of $V$ allows to obtain not only the optimal(s) $r^*$ at time zero, but a whole family of such optimal cash-allocations depending on time and the extended state space variables.
\end{remark}

\section{Examples and classical solutions}
\label{sec examples}

 In this part we solve the HJB equation \eqref{eq:hjb} for specific OCE risk measures, giving us the chance to apply Theorem \ref{thm:exist} and providing examples. At the same time we shall seek conditions on the data of the problem in order to guarantee that \eqref{eq:hjb} has a classical solution. 
Of the examples we look at, only the entropic risk measure is time-consistent. For simplicity, we assume $m =d =1$ throughout.

\subsection{Entropic risk measure}
\label{sec:entropic}
For $$l(x)=e^{x}-1\mbox{, so }l^*(z)=z\log z-z+1,$$ we get the entropic risk measure $$\rho(X)=\log E[e^{X}].$$
In our language, we easily obtain 
$$\rho^{l_z}(zX)=z\log\left (\frac{ E[e^{X}]}{z}\right )+z - 1,$$
and so for the value function, see Proposition \ref{prop:struct-V}, we have
\begin{align}\notag
V(s,y,z)&\textstyle =-z\log z +z - 1+ z\log E \left [ \exp\left  \{  f(Y_T^{s,y}) + \int_s^T g(t,Y_t^{s,y})ds  \right\} \right ] 
=: -z\log z+z - 1 + z \tilde{V}(s,y).\notag
\end{align} 

\begin{proposition}
Under (A1)-(A2) the function $\exp{(\tilde{V})}$ is a viscosity solution of the backward Kolmogorov PDE associated to the diffusion $Y$, the discount/killing rate $g$ and the final condition $\exp(f)$. Assuming that $\sigma^2>\varepsilon$ everywhere for some $\varepsilon>0$ (uniform parabolicity), and that $b,\sigma^2$ are bounded, we have that $\exp{(\tilde{V})}$ is the unique classical solution of such PDE. Correspondingly, $V$ is the classical solution of our HJB equation under these conditions, and is of class $C^{1,2,2}$ at least.
\end{proposition}
\begin{proof}
The first statement follows e.g.\ from  \citet[Chap. V.9]{Flem-Soner-second}. The second by e.g.\ \citet[Theorems 1.7.12 and 2.4.10]{Friedman-book}. It is then clear that $V$ is the classical solution of our HJB equation.  
\end{proof}
That $V$ is a classical solution can also be obtained without the uniform parabolicity condition, provided one assumes further smoothness of $b,\sigma$ and $f$. This is proved by stochastic flows techniques, as in \citet[Chap. V.3]{IW2nd}.

\subsection{Monotone mean-variance}

Here $$\textstyle l(x)=\frac{(1+x)_+^2-1}{2},$$ with $(x)_+:= \max(x,0)$. Therefore $l^*(z)=\frac{(z-1)^2}{2}$ on $[0,\infty)$ and equal to $+\infty$ otherwise.
In this case, the corresponding OCE is the so-called monotone mean-variance risk measure referred to in the introduction. Formally, the HJB equation \eqref{eq:hjb} becomes
	\begin{align}\label{eq formal HJB}
	\begin{cases}
	&\partial_tV + b\partial_yV+\frac{1}{2}\sigma^2\left[\partial^2_{yy}V-\frac{[\partial^2_{yz}V]^2}{\partial^2_{zz}V}\right] +zg= 0,\quad (s,y,z)\in [0,T)\times \mathbb{R}\times \mathbb{R}_+\\
	&V(T,y,z) = f(y)z-l^*(z),\quad (y,z)\in \mathbb{R}\times  \mathbb{R}_+,
	\end{cases}
	\end{align}
after solving the scalar quadratic maximization problem (a concave one, since formally $\partial^2_{zz}V\leq 0$) therein. 
We make the educated guess
\begin{align}
V(t,y,z)=\phi(t,y)+z\tilde{V}(t,y)-l^*(z).\label{eq Ansatz quadratic}
\end{align}
For this to be true, and assuming for a moment enough smoothness, it is necessary that
\begin{align}\label{eq MMV Vtilde}
\left\{ 
\begin{array}{ccc}
\left(\partial_t + b\partial_y +\frac{1}{2}\sigma^2 \partial^2_{yy}\right)\tilde{V} &= &-g \\
\tilde{V}(T,\cdot) &=& f(\cdot),
\end{array}
\right . 
\end{align}
as well as 
\begin{align}
\label{eq MMV phi}
\left\{ 
\begin{array}{ccc}
\left(\partial_t + b\partial_y +\frac{1}{2}\sigma^2 \partial^2_{yy}\right)\phi &= & -\frac{1}{2}\sigma^2[\partial_y \tilde{V}]^2 \\
\phi(T,\cdot) &=& 0,
\end{array}
\right .
\end{align}
as can be readily verified by plugging the Ansatz in \eqref{eq formal HJB}. We have
\begin{proposition}
\label{prop mmv}
Assume (A1)-(A2), that $\sigma^2>\varepsilon$ everywhere (uniform parabolicity), and that $b,\sigma^2$ are of class $C^{1,3}$, g is $C^{1,2}$ and $f$ is $C^2$, all of them bounded with bounded derivatives (uniformly in time, when applicable). Then equations \eqref{eq MMV Vtilde}-\eqref{eq MMV phi} have unique classical solutions, and the HJB equation \eqref{eq:hjb} has a unique solution of the form \eqref{eq Ansatz quadratic}. This solution is equal to the value function. 
\end{proposition}
\begin{proof}
 By \citet[Theorems 1.7.12 and 2.4.10]{Friedman-book}, equation \eqref{eq MMV Vtilde} has a unique classical solution  $\tilde{V}$. There are a number of ways to obtain that $\tilde{V},\partial_y\tilde{V}$ and $\partial^2_{yy}\tilde{V}$ are bounded. For instance, one can after differentiate \eqref{eq MMV Vtilde} twice with respect to $y$, and apply the parabolic maximum principle. Alternatively, one can use the Feynman-Kac representation of $\tilde{V}$ and stochastic flows techniques to represent, and bound, these derivatives. In any case, the term $\sigma^2[\partial_y\tilde{V}]^2$ in the r.h.s.\ of \eqref{eq MMV phi} becomes in particular Lipschitz, so applying \citet[Theorem 1.7.12 and Theorem 2.4.10]{Friedman-book} again, we get that equation \eqref{eq MMV phi} has a unique classical solution $\phi$. Thus the HJB equation \eqref{eq:hjb} also has a classical solution, which by construction has the form \eqref{eq Ansatz quadratic}. We argue that this solution, which we now call $v$, equals the value function. Since $v$ is a supersolution, we have by Theorem \ref{thm:exist} that it is no smaller than the value function. 
 The converse inequality (actually, the full equality) can be obtained by verification as follows. First, by It\^o formula (for $\beta \in \mathcal{L}_b$) and the HJB equation (as $Z^\beta$ takes values in $(0,\infty)=\mathcal{O}$),
 \begin{align*}
& \textstyle E[ v(s+u, Y_{s+u}^{s,y},Z_{s+u}^{s,z,\beta}  )-v(s,y,z)]\\ = & 
\textstyle 
E\left[\int_s^{s+u}\left ( \partial_t+b(t,Y_{t}^{s,y})\partial_y +\frac{1}{2}[\beta^2_t(Z_{t}^{s,z,\beta} )^2\partial^2_{zz} + \sigma^2(t,Y_{t}^{s,y})\partial^2_{yy}]+\beta_t Z_{t}^{s,z,\beta} \sigma(t,Y_{t}^{s,y})\partial^2_{yz}\right )v(t, Y_{t}^{s,y},Z_{t}^{s,z,\beta} )dt \right ]
\\ \leq &\textstyle E\left[\int_s^{s+u} -Z_{t}^{s,z,\beta} g(t,Y_{t}^{s,y}) dt\right ],
  \end{align*}
 from which we obtain one-half of the DPP for the classical solution of the HJB equation
 	\begin{equation}\textstyle
	\label{eq:half Bellman}
 		v(s,y,z)\geq\sup\limits_{\substack{\beta \in \mathcal{L}_b}} E\left[  \int_s^{s+u} g(t,Y^{s,y}_t)Z_t^{s,z,\beta}dt + v\left(s+u, Y_{s+u}^{s,y},Z_{s+u}^{s,z,\beta} \right ) \right], \,\, \mbox{ }y\in {\mathbb R^m},z>0,s+u\leq T.
	\end{equation}
	Formally solving the maximization problem in the HJB, we guess that $\bar{\beta}(t,y,z)=-\sigma\frac{\partial^2_{yz}v}{z\partial^2_{zz}v}(t,y,z)$ should provide an optimal (Markov) control, if only $\partial^2_{zz}v$ did not vanish and the associated $Z^{\bar{\beta}}$ was a well-defined martingale and never touched zero. Since $v$ has the form \eqref{eq Ansatz quadratic} we get $\partial^2_{zz}v=-1$ and $\bar{\beta}(t,y,z)=-\sigma\partial_y \tilde{V}(t,y) /z$. As we have observed, $z\bar{\beta}(t,y,z)$ is bounded. Consequently $dZ^{\bar{\beta}}_t=-\bar{\beta}(t,Y_t)Z^{\bar{\beta}}_tdW_t$ defines a true martingale, which never touches zero. This can be used to prove that \eqref{eq:half Bellman} is an equality, and evaluating this equality at $u=T-s$ yields $V=v$ everywhere, i.e.\ the classical solution of the HJB equation is the value function.
\end{proof}

Proposition \ref{prop mmv} shows that $\rho(X^y) = V(0,y,1)$ with $V$ given by \eqref{eq Ansatz quadratic} and $X^y$ defined in Theorem \ref{thm:exist}. 
Putting $\tilde{Y}^y:=\tilde{V}(t,Y^{s,y}_t)$, due to \eqref{eq MMV Vtilde} and an application of It\^o's formula shows that there exists $\tilde{Z}^y \in {\cal L}^2$ such that 
\begin{equation}
\label{eq:bsde1}
d\tilde{Y}^y_t = -g(t,Y^{s,y}_t)dt+\tilde{Z}^y_t\,dW_t, \quad\text{with }\tilde{Y}^y_T = f(Y^{s,y}_T).
\end{equation}
Similarly, putting $\hat{Y}^y_t:= \phi(t,Y^{s,y}_t)$, there is $\hat{Z}^y \in {\cal L}^2$ such that
\begin{equation}
\label{eq:bsde2}
	d\hat{Y}^y_t = - \frac{1}{2}\sigma^2(Y^{s,y}_t)\partial_y\tilde{Y}^y_t\,dt + \hat{Z}^y_t\,dW_t \quad \text{with } \hat{Y}_T^y=0.
\end{equation}
Therefore, translated into the language of backward stochastic  differential equations, Proposition \ref{prop mmv} shows that in the Markovian case the (time-inconsistent) monotone mean-variance risk measure satisfies the representation 
$$\rho(X) = \hat{Y}^{y}_0 + \tilde{Y}^{y}_0$$ where $\tilde{Y}^{y}$ and $\hat{Y}^y$ are solutions of the (embedded) backward equations \eqref{eq:bsde1} and \eqref{eq:bsde2}.
This is in sharp contrast with the time-consistent case discussed in the introduction and in Section \ref{sec:entropic}.

The reader may wonder, as we did, why the simple structure \eqref{eq Ansatz quadratic} arises at all. Namely, we ask: is there a pure probabilistic/optimization argument justifying \eqref{eq Ansatz quadratic}? 
Here is a result in this direction, but it requires a technical assumption which we discuss after the ensuing proof:
\begin{itemize}
\item[(\textbf{P})] If $m(s,{y};T,dx)$ denotes the Markov transition kernel from $Y_s=y$ to $Y_T\in dx$, then for any $s<T$ and $y$ we have $ m(s,y;T,dx)\gg P(Y^{s,y}_T\in dx)$.
\end{itemize}

\begin{proposition}
Assume that $b,\sigma,f$ satisfy our standard assumptions (A1)-(A2) and are further $C^2$ is space with bounded derivatives. Assume also that (\textbf{P}) holds, and for simplicity that $g\equiv 0$. 
Then the value function has the structure \eqref{eq Ansatz quadratic}. 
\end{proposition}

\begin{proof}\\
For ease of notation denote $L= b\partial_y +\frac{1}{2}\sigma^2 \partial^2_{yy}$, so by It\^o's formula:
\begin{multline}\textstyle
V(s,y,z)=f(y)z-l^*(z)+\\ \textstyle \sup_{\beta\in\mathcal{L}_b}E\left [ Z^{s,z,y,\beta}_T\int_s^T Lfdt  + \int_s^T \left\{\beta_tZ^{s,z,y,\beta}_t(\sigma f') - (\beta_t Z^{s,z,y,\beta}_t)^2/2 \right \} dt 
\right ].\notag
\end{multline}
The existence of an optimal $\beta$ follows from \citet[Proposition 1.3]{Dra-Kup-Pap} (implying the existence of an optimal and essentially bounded $Z$) and martingale representation. The corresponding $\beta$ may not be essentially bounded, but this is irrelevant here. We write now the Pontryagin principle (necessary optimality conditions, see \citet[Chap. 3.3]{YongZhou}) for the above problem: there exists $(p^{s,y,z},q^{s,y,z})$ such that
\begin{align}
\label{eq pontryagin}
\left\{ 
\begin{array}{ccl}
dp^{s,y,z}_t &= & \beta_t\left(q^{s,y,z}_t  +(\sigma f')(t,Y_t^{s,y}) - \beta_t Z^{s,z,y,\beta}_t \right)dt+q^{s,y,z} dW_t\\
p_T^{s,y,z} &=& - \int_s^T Lf(t,Y_t^{s,y})dt  ,
\end{array}
\right .
\end{align}
and for the optimal $\beta^{s,y,z}_t$ we have that it maximizes ($t\geq s$)
$$v\mapsto q^{s,y,z}_t v Z_t + vZ_t(\sigma f')(t,Y_t^{s,y}) - (v Z_t)^2/2 ,$$
i.e.\ $\beta^{s,y,z}_t(Z^{s,z,y,\beta}_t)^2 = Z^{s,z,y,\beta}_tq^{s,y}_t+Z^{s,z,y,\beta}_t(\sigma f')(t,Y_t^{s,y})$. Strictly speaking there should be also a pair of co-states/multipliers $(\hat{p},\hat{q})$ associated to the drift and volatility of $Y$, but since these are independent from both $Z$ and $\beta$ we can clearly ignore $(\hat{p},\hat{q})$. We also observe that \citet[Theorem 3.2, Chap. 3.3]{YongZhou} was applicable thanks to the fact that a fortiori the process $Z$ is essentially bounded; otherwise the Lipschitz condition needed there would fail. {Assuming for the time being that the optimal $Z^{s,z,y,\beta}_t$ is a.s.\ strictly positive on $[s,T)$}, so the previous scalar first order condition implies that the drift part in \eqref{eq pontryagin} vanishes identically, we obtain that $(p^{s,y,z},q^{s,y,z})=(p^{s,y},q^{s,y})$, i.e.\ neither of them depend on $z$.
Furthermore, we have $$dZ^{s,z,y,\beta}_t = Z^{s,z,y,\beta}_t\beta^{s,y,z}_t dW_t =[q^{s,y}_t+(\sigma f')(t,Y_t^{s,y})]dW_t,$$ so 
we finally get
\begin{align*}
V(s,y,z)& \textstyle =f(y)z-l^*(z)+zE\left [  \int_s^T Lf(t,Y_t^{s,y})dt  \right ]\\ & \textstyle
+E\left [\int_s^T \left \{ \int_s^t [q^{s,y}_r+(\sigma f')(r,Y_r^{s,y})]dW_r\right \}Lf(t,Y_t^{s,y})dt \right ]\\
& \textstyle +E\left [\int_s^T [q_t^{s,y} + \sigma f'](\sigma f')(t,Y_t^{s,y})dt\right ]
-\frac{1}{2}E\left [\int_s^T [q_t^{s,y} + \sigma f'(t,Y_t^{s,y})]^2dt\right ],
\notag
\end{align*}
which indeed has the form \eqref{eq Ansatz quadratic}.

Now observe that at the optimum we must have $Z^{s,z,y,\beta}_T = F(Y_T^{s,y})$ for some Borel non-negative bounded function $F$; indeed, by Jensen's inequality, projection can only reduce a convex expected-type cost. 
From this we see by the Markov property that for $s\leq t<T$,
$$\textstyle 
Z^{s,z,y,\beta}_t = E[F(Y^{s,y}_T)|Y^{s,y}_t]=\int F(x)m(t,Y^{s,y}_t;T,dx).
$$
Observe that (\textbf{P}) is, by the flow property of SDEs, actually equivalent to: for any $s\leq t<T$ it holds $a.s. \,\, m(t,Y^{s,y}_t;T,dx)\gg P(Y^{s,y}_T\in dx)$. This and the above equality gives that $
Z^{s,z,y,\beta}_t >0$ a.s.\ since evidently $P(F(Y^{s,y}_T)>0)>0$ by $E[F(Y^{s,y}_T)]=1$.

\end{proof}
Assumption (\textbf{P}) is fulfilled if $\sigma,b$ are functions of time at most, with $\sigma^2>0$ and $\int_0^T\sigma^2(t)dt<\infty$, since then the pairs $(Y_t,Y_T)$ are non-degenerate bivariate Gaussians. {More generally, (\textbf{P}) is satisfied if the transition kernel of $Y$ is everywhere equivalent to a fixed reference measure. This holds in the presence of uniform parabolicity, plus enough smoothness and boundedness of the coefficients, and the reference measure is then Lebesgue; see \citet{KusuokaStroockpositivity} for the original Malliavin calculus approach to this issue, and e.g.\ \citet[Theorem 2.3]{Nualartshort} for a sample result.

\subsection{Conditional Value-at-Risk}
\label{sec CVAR}

Here $$l(x) = x^+/\alpha,$$ for some $\alpha\in (0,1)$. Thus $l^*(z) = 0$ if $z\in [0,1/\alpha]$ and $+\infty$ else. The OCE so defined,$$\rho(\cdot)=:\text{CVaR}_\alpha (\cdot) ,$$ is called conditional value-at-risk (alternatively tail / average value-at-risk, or expected shortfall) at level $\alpha$, and is widely used in practice.
In this setting, Proposition \ref{prop:struct-V} in conjunction with Theorem \ref{thm:exist} state that 
\begin{equation} \label{eq Cvar}
V(s,y,z)\,\, := \,\, \text{CVaR}_{\alpha z}\Bigl( zf(Y_T^{s,y})+z\int_s^T g(t,Y_t^{s,y})dt \Bigr ),
\end{equation}
is a viscosity solution, and minimal viscosity supersolution, of the HJB equation \eqref{eq:hjb} on $[0,T]\times \mathbb{R} \times (0,1/\alpha)$, with terminal condition $V(T,y,z)=zf(y)$.
In addition, it is easily checked that $V$ satisfies the boundary conditions
\begin{equation}\label{eq boundary}\textstyle
V(s,y,1/\alpha)=\frac{1}{\alpha}E\left [f(Y_T^{s,y})+\int_s^T g(t,Y_t^{s,y})dt  \right ] \quad \text{and} \quad V(s,y,0) = 0\quad \text{for all $(s,y)$}.
\end{equation}
 This is in line with the discrete-time analogue studied by \citet{PflugPichlerinconsistency1}. 
 
 Recall that $V(s,y,1)$ is the CVaR at level $\alpha$ of $$X^{s,y}:=f(Y_T^{s,y})+\int_s^T g(t,Y_t^{s,y})dt.$$
 We think it is very noteworthy that the CVaR can be characterized through PDE methods.
Our approach further allows to characterize the celebrated Value-at-Risk, defined (in the current interpretation of $X$ as a \emph{loss}) as
\begin{equation*}
	\text{VaR}_\alpha(X^{s,y}):= \inf\{m\in \mathbb{R}:P(X^{s,y}> m)\le \alpha\}.
\end{equation*}
In fact, since VaR is positive-homogeneous, it follows by \cite[Lemma 4.51]{FS3dr} that
\begin{align*}
	V(s,y,z) &=\text{CVaR}_{\alpha z}(zX^{s,y}) = \frac{1}{\alpha z}\int_0^{\alpha z}\text{VaR}_u(zX^{s,y})\,du
	         = \frac{1}{\alpha}\int_0^{\alpha z}\text{VaR}_u(X^{s,y})\,du.
\end{align*}
Thus, $V$ is differentiable in $z$ and we have $$\partial_zV(s,y,z) = \text{VaR}_{\alpha z}(X^{s,y}),$$
and in particular
$$\textstyle
\text{VaR}_\alpha\left(f(Y_T^{s,y})+\int_s^T g(t,Y_t^{s,y})dt \right) = \partial_zV(s,y,1).$$

We conjecture that the HJB equation for conditional value-at-risk cannot be generally solved via separation of variables arguments as we did in the cases of monotone mean-variance and entropic risk measures. Moreover we also conjecture that scalar perturbations, and their combinations, of these two cases are the only OCEs for which separation of variables can generally work. A corollary of these conjectures is that the quest for classical solutions is much harder for OCEs which are not some sort of perturbation/combination of monotone mean-variance and entropic risk measures. In particular, resorting to viscosity solutions seems to be unavoidable for the PDE proposed.

\begin{remark} \label{rmk final uniqueness}A final word on uniqueness. The ``parabolic boundary conditions'' \eqref{eq boundary} are by definition satisfied by the value function. They seem to be valid also for the value functions associated to other OCE risk measures for which the domain of $l^*$ is bounded. In such a situation, our main result Theorem \ref{thm:exist} pins down these value functions as the minimal supersolution of our corresponding HJB equations, even though it is not difficult to construct multiple solutions to these PDEs. If one wanted to develop an existence theory of our HJB in the bounded-domain case, one would therefore have to add the parabolic boundary conditions into the mix. As we have mentioned, at the moment we do not have a comparison principle at hand, and so we do not pursue this line of thought in the present work.
\end{remark}

\section{Extensions}
\label{sec: extensions}

We first explain how our results can be leveraged to cover the class of ``utility-based expected shortfall risk measures.'' This already shows the limits of the PDE approach. Next, we describe how our approach can be applied to general, not necessarily Markovian claims. We finally provide for convenience of the reader the non-local PDE that appears if instead of Brownian-driven diffusions, the claim $X$ was written on a jump-diffusion model, as described in the introduction.

\subsection{Beyond OCE risk measures: utility-based expected shortfall}
\label{sec: non OCE}
The expected shortfall, also known as utility-based shortfall risk measure, was introduced by \citet{FS3dr} and is defined (for $\delta$ a fixed threshold in the interior of $range(l)$) as
\begin{equation*}
 	\rho^{ES}(X): = \inf\{r\in \mathbb{R}: E[l(-r+X)]\le \delta\}, \quad X \in L^\infty.
 \end{equation*}
 This is in general not an OCE.
 However, its dynamic representation can be derived from our study.
 Recall that a function $f:\mathbb{R}^p\to\mathbb{R}$ is said to be positive homogeneous if $f(\lambda x)=\lambda f(x)$ for all $\lambda>0$ and $x \in \mathbb{R}^p$.
 The positive homogeneous upper envelope of a function $f$ is the smallest positive homogeneous function that lies above $f$.

\begin{lemma}
\label{lem:ES}
Assume that (A1)-(A2) hold and $l$ is as before. Put $X^y = f(Y^{0,y}_T) + \int_0^Tg(t, Y_t^{0,y})\,dt$. Then we have $$\rho^{ES}(X^y)={\cal V}(0,y,1),$$ where 
\begin{equation}\label{eq cal V}\textstyle
	{\cal V}(s,y,z) := \sup\limits_{\lambda>0}\sup\limits_{\substack{\beta \in \mathcal{L}_b}} E\left[ f(Y^{s,y}_T)Z_T^{s,z,\beta} - \frac{1}{\lambda}\left({l}^*\left (\lambda Z_T^{s,z,\beta}\right )+ \delta\right) +\int_s^T g(t,Y^{s,y}_t)Z_t^{s,z,\beta}dt \right].
\end{equation}
 Letting $V$ be defined as in Proposition \ref{prop:struct-V}, i.e.\ the value function for the OCE with loss function $l$, 
we have that ${\cal V}(s,y,\cdot)$ is furthermore the positive homogeneous upper envelope of $V(s,y,\cdot)-\delta$, and can be obtained via
\begin{equation}
	\label{eq:hullV}
	{\cal V}(s,y,z) = \sup_{\lambda >0}\frac{1}{\lambda}\left(V(s,y,\lambda z)-\delta\right).
\end{equation}
\end{lemma} 

\begin{proof}
By \citet[Theorem 4.115]{FS3dr} we have
 \begin{equation*}\textstyle
 	\rho^{ES}(X^y) = \sup_{Z \in L^1_+}\left(E[ZX^y] - \inf_{\lambda>0}\frac{1}{\lambda}\left(\delta + E[l^*(\lambda Z)]\right) \right).
 \end{equation*}
It follows from the dual representation of OCEs that 
$	\rho^{ES}(X^y) = \sup_{\lambda>0}\rho^\lambda(X^y)$,
where $\rho^\lambda$ is the OCE corresponding to the loss function $l^\lambda(x):=(l(x)-\delta)/\lambda$.
Thus by definition $\rho^{ES}(X^y) = {\cal V}(0,y,1)$ with $\cal V$ as in \eqref{eq cal V}. Since $\lambda Z^{s,z,\beta}=Z_t^{s,\lambda z,\beta}$, the definition of $\cal V$ immediately leads to \eqref{eq:hullV}.
It is not hard to verify that this is the smallest positive homogeneous function (in $z$) which dominates $V(s,y,\cdot)-\delta$.
\end{proof}
 
At this stage, we reasonably ask if $\cal V$ can be characterized by PDE arguments, or what is almost the same, if there is a dynamic programming principle for it. Here the story is very different from the OCE case. Since $z\mapsto {\cal V}(s,y,z)$ is positive homogeneous, we have $${\cal V}(s,y,z)= z \hat{V}(s,y),$$
with $\hat{V}(s,y):={\cal V}(s,y,1)$. Let us for the sake of the argument assume that $\hat{V}$ is smooth. It is not difficult to see that $\cal V$ must satisfy the same HJB equation \eqref{eq:hjb} we encountered in the OCE case, but with the terminal condition $\Psi(y,z):=f(y)z-\inf_{\lambda>0}(\frac{1}{\lambda}(l^*(\lambda z)-\delta))$, and we shall prove this shortly. But then the l.h.s.\  of \eqref{eq:hjb} becomes 
 	\begin{align*}
		z\left(\partial_t\hat{V}+b\partial_y\hat{V}+\frac{1}{2}tr\left(\sigma\sigma'\partial^2_{yy}\hat{V}\right)\right )
		+\sup_{\beta\in\mathbb{R}^d}\left [z \,\partial_{y}\hat{V} \sigma\beta\right ] + zg(s,y),
	\end{align*}
 which easily blows up to $+\infty$ unless $\sigma$ is trivial or $f$ is a constant. Hence a smooth $\cal V$ is easily a subsolution and basically never a supersolution of  \eqref{eq:hjb} in the sense we have considered. We conjecture, but do not pursue this in the present work, that a better way to characterize $\cal V$ is through variational inequalities. 
 
 Building on these observations, we now prove the subsolution property of $\cal V$ in the non-smooth case:
 \begin{theorem}
 \label{prop:ES}In the setting of Lemma \ref{lem:ES}, the function 
${\cal V}$ is a viscosity subsolution of the PDE \eqref{eq:hjb} where the terminal condition $\psi$ is replaced by $\Psi(y,z):=f(y)z-\inf_{\lambda>0}(\frac{1}{\lambda}(l^*(\lambda z)-\delta))$.
\end{theorem}

\begin{proof}
For every $\lambda>0$, the function $V^\lambda(s,y,z):= (V(s, y, \lambda z) - \delta)/\lambda$ is a viscosity subsolution of the HJB equation \eqref{eq:hjb} where $\psi$ is replaced by $\psi^\lambda(y,z):= f(y)z - [l^*(\lambda z)+\delta]/\lambda$.
As a matter of fact, let $\varphi \in C^2$ be a test function such that $V^\lambda -\varphi$ has a maximum at $(s_0,y_0,z_0)\in [0,T]\times \mathbb{R}^m\times {\cal O}$. 
Assume that $s_0<T$.
Putting $\varphi^\lambda(s, y, z):= \lambda\varphi(s, y, z/\lambda) + \delta$, it follows that $(s_0, y_0,\lambda z_0)$ is a maximum of $V-\varphi^\lambda $.   
It follows from Theorem \ref{thm:exist} that
\begin{equation*}
\partial_t\varphi^\lambda(s_0, y_0, \lambda z_0) + H(s_0, y_0, \lambda z_0, D\varphi^\lambda(s_0, y_0, \lambda z_0), D^2\varphi^\lambda(s_0, y_0, \lambda z_0))\ge 0,
\end{equation*}
which easily yields 
\begin{equation*}
\partial_t\varphi(s_0, y_0, z_0) + H(s_0, y_0, z_0, D\varphi(s_0, y_0, z_0), D^2\varphi(s_0, y_0, z_0))\ge 0.
\end{equation*}
If $s_0=T$, then one has $\varphi(T, y_0, z_0) = (V(T, y, \lambda z) - \delta)/\lambda = f(y)z -(l^*(\lambda)+\delta)/\lambda$. 

	We can use a stability argument similar to the one in the proof of Theorem \ref{thm:exist} to show that the pointwise supremum ${\cal V}= \sup_{\lambda >0}V^\lambda$ is a viscosity subsolution of the HJB equation \eqref{eq:hjb} where the terminal condition $\psi$ is replaced by $\Psi$.
	To wit, let  $\varphi \in C^2$ be a test function such that ${\cal V} -\varphi$ has a strict local maximum at $x_0=(s_0,y_0,z_0)\in [0,T)\times \mathbb{R}^m\times {\cal O}$ and $(x_0, D\varphi(x_0), D^2\varphi(x_0))\in \text{int dom(H)}$.
	Let $B_r(x_0):=\{x: |x-x_0|\le r\}$ be a ball on which $ {\cal V}(x_0) - \varphi(x_0)> {\cal V}(x)-\varphi(x)$ holds.
	Since ${\cal V}(x_0)=\sup_{	\lambda > 0}V^\lambda(x_0)$, for all $n \in \mathbb{N}$ there is $\lambda^n$ such that ${\cal V}(x_0) \le V^{\lambda^n}(x_0) +1/n$.
	Let $x^n$ be a point at which $V^{\lambda^n}-\varphi$ reaches its maximum in $B_r(x_0)$.
	Up to a subsequence, $(x^n)$ converges to some $x\in B_r(x_0)$ and since $V^{\lambda^n}\le {\cal V}$ for all $n$, one has
	\begin{equation*}
		 {\cal V}(x^n) - \varphi(x^n) \ge V^{\lambda^n}(x^n) - \varphi(x^n) \ge V^{\lambda^n}(x_0) - \varphi(x_0) \ge {\cal V}(x_0) - \varphi(x_0) - 1/n.
	\end{equation*}
	Passing to the limit as $n$ goes to infinity gives ${\cal V}(x)-\varphi(x) \ge {\cal V}(x_0)-\varphi(x_0)$, showing that $x= x_0$.
	Choose $n$ large enough such that $(x^n, D\varphi(x^n), D^2\varphi(x^n))\in \text{int dom(H)}$, which is possible since $\varphi \in C^2$.
	Since $V^{\lambda^n}$ is a viscosity subsolution, it holds 
	\begin{equation*}
\partial_t\varphi(x^n) + H(x^n, D\varphi(x^n), D^2\varphi(x^n))\ge 0
\end{equation*}
so that taking the limit as $n$ goes to infinity gives by continuity of $H$ in $\text{int dom(H)}$
	\begin{equation*}
\partial_t\varphi(x_0) + H(x_0, D\varphi(x_0), D^2\varphi(x_0))\ge 0.
\end{equation*}
If $s_0 = T$, then $\varphi(T, y_0,z_0)= {\cal V}(T, y_0,z_0)=\Psi(y_0,z_0)$.
\end{proof}

This should serve as a cautionary tale, if one wanted to extend the results we have obtained beyond the class of OCEs. Already for utility-based expected shortfalls, which are so closely related to OCEs, things can be very different.

\subsection{On general, not necessarily Markovian claims}
\label{sec: heuristic general claims}

We have so far shown how to compute the risk of Markovian claims. Such claims could be truthfully described as ``static'', being limits of claims of the form $\sum_{i=0}^k h_i(Y_{t_i})$. We now briefly describe how more general path-dependent claims could be embedded in our framework, the key here being the law-invariance of OCE's. This is only a proof of concept, so we do not work out the details. For simplicity, we take $W$ to be a one-dimensional Brownian motion ($m=d=1$).

Given $X\in L^{\infty}$ a non-trivial claim with distribution function $F_X(dx)$, and denoting $F_N(dx)$ the distribution function of a centred Gaussian with variance $T$, it is well-known that $H:=F_X^{-1}\circ F_N$ is the unique right-continuous non-decreasing mapping such that $H(W_T)$ is distributed like $X$. 
Accordingly $$\rho(H(W_T))=\rho(X),$$ by law-invariance. Denote by $Y_t:=E[H(W_T)|{\cal F}_t]$ the so-called Bass martingale (see \citet{BassMartingale}). It is elementary that $Y_t=u(t,W_t)$ for $u(t,x)=\int H(x+z)d\gamma_{T-t}(dz)$ with $\gamma_{s}$ the centred Gaussian measure with variance $s$. Since $H$ is non-constant and increasing, it clearly follows that $u(t,\cdot)$ is strictly increasing. In particular defining $v(t,\cdot):=u(t,\cdot)^{-1}$ the inverse function of $u(t,\cdot)$, we obtain
\begin{equation}
dY_t = \partial_xu(t,v(t,Y_t))dW_t, \label{PDE Bass}
\end{equation} 
so we get a martingale diffusion such that $\rho(Y_T)=\rho(X)$, and modulo the technical assumptions on the volatility above we are back in our framework with Markovian claims. 

An alternative idea is to look for a diffusion with unit volatility having the desired distribution at time $T$. For instance, if $X$ has a density $f_X(\cdot)$, then the solution of 
$$d Y_t = \frac{\partial_x w}{w}(t,Y_t)dt+dW_t,$$ 
satisfies $\mbox{Law}(Y_T) =\mbox{Law}(X) $, provided $(\partial_t+1/2\partial^2_{xx})w=0$ and $w(T,\cdot)=f_X(\cdot)$.

Without this ``Bass martingale'' argument, or the unit-volatility diffusion idea, a direct dynamic programming approach would seem to require either more advanced semigroup arguments and/or stochastic PDEs and/or path-dependent PDEs.

\subsection{Jump-diffusion models and a non-local HJB equation}
\label{subsection Jump case}

We show, specializing the setting of \citet{Oksendaltrivial}, that even in the presence of jumps there is still a dynamic representation of OCEs (for Markovian claims) to be expected. Unlike in the Brownian framework, we will not prove the validity of this representation. 

Suppose that the ambient filtration is generated by a $d$-dimensional standard Wiener processes $W$ as before, plus $\ell$ independent compensated Poisson random measures $\{\tilde{N}^i(dt,d\xi)\}_{i=1}^\ell$ on $[0,T]\times \mathbb{R}^\ell$. Equivalently, if $\{{N}^i(dt,d\xi)\}_{i=1}^\ell$ are given independent jump measures and $\nu^i(d\xi):= E[N^i((0,T]\times d\xi)]$ we take $\tilde{N}^i(dt,d\xi)= {N}^i(dt,d\xi)-\nu^i(d\xi)dt$. We consider a Markovian claim/position $X$ as in \eqref{def markovian claim}, but written in terms of the jump diffusion $Y$ defined by
\begin{align*}
\textstyle dY^{s,y}_t &= b(t,Y^{s,y}_{t-})dt+\sigma(t,Y^{s,y}_{t-})dW_t+ \int_{\mathbb{R}^\ell} \gamma(t,Y^{s,y}_{t-},\xi)\tilde{N}(dt,d\xi), \quad t\ge s\\\textstyle 
Y^{s,y}_s&= y,
\end{align*}
where now $\gamma$ is $\mathbb{R}^{m\times\ell}$-valued and satisfies suitable assumptions.
In this filtration one can represent reasonable change of measures via densities $Z^{\beta,\theta}$ satisfying
$$\textstyle dZ^{\beta,\theta}_t= Z^{\beta,\theta}_{t-}\left[ \beta_t dW_t +\sum_{i=1}^\ell \int_{\mathbb{R}^\ell}\theta^i(t,\xi)\tilde{N}^i(dt,d\xi) \right].$$
 Plugging this into the dual representation of OCEs, one obtains again a stochastic control representation of the risk $\rho(X)$. At a formal level, we would expect to obtain under suitable assumptions:
 \medskip
 
\begin{conjecture} For the Markovian claim $X=X^y$ we have $$\rho(X)=V(0,y,1),$$
where $V$ is the minimal viscosity solution of the integro-partial differential equation: 
	\begin{align*}
		&\partial_tV(s,y,z)+b(s,y)\partial_yV(s,y,z)+\frac{1}{2}tr\left(\sigma(s,y)\sigma(s,y)'\partial^2_{yy}V(s,y,z)\right)\\&+\sup_{\theta\in\mathbb{R}^\ell}\Bigl (\sum_{i=1}^\ell\int_{\mathbb{R}^\ell}\Bigl\{V(s,y+\gamma^{(i)}(s,y,\xi),z+z\theta^i)  -V(s,y,z) -\partial_y V(s,y,z) \gamma^{(i)}(s,y,\xi)   \\& -\partial_z V(s,y,z) z\theta^i\Bigr \}\nu^i(d\xi)\Bigr ) 
		+\sup_{\beta\in\mathbb{R}^d}\left (\frac{1}{2} z^2|\beta|^2\partial^2_{zz}V+ z\, \partial^2_{yz}V\sigma(s,y)\beta\right) + zg(s,y) =0,
	\end{align*}
with terminal condition
$$V(T,y,z)= f(y)z - l^*(z).$$
\end{conjecture}

Notice the presence of the non-local term when taking supremum over $\theta$. This conjecture is open, as far as we know. In \citet{Oksendaltrivial} the authors study a more complicated non-local Hamilton-Jacobi-Bellman-Isaacs PDE related to a risk minimization problem, but start by assuming that such PDE has a classical solution (verified in the examples therein). As we have seen, already in the absence of jumps the viscosity solution approach is unavoidable and requires a good deal of work, due to the singularity of the Hamiltonian. In the case with jumps we expect that similar (if more delicate) arguments as the ones we have employed should deliver a positive answer to this conjecture. 

\medskip

\textbf{Acknowledgements:}
We thank Beatrice Acciaio, Joaqu\'in Fontbona, Asgar Jamneshan and Michael Kupper for their feedback on this article.

\section*{Appendix}

\medskip

\begin{proof}(of Lemma \ref{thm:rep})\\
A derivation of the dual representation
\begin{equation}\textstyle 
	\rho(X)\,=\, 	\sup\left \{ E[XZ] -E[l^*( Z)]:Z\in L^{1}_{+} ,Z\in \text{dom}(l^*),\,E[ Z]=1  \right\},\quad X \in L^\infty\label{dual_rep}
\end{equation}
 can be obtained for instance from \citet{ben-tal01}, and elementary considerations.
Deriving \eqref{dual_rep1} from \eqref{dual_rep} is done by classical arguments which we give for completeness.
We clearly have ``$\geq$'' in \eqref{dual_rep1}. 
Conversely, given $\varepsilon>0$ and a feasible $Z$ for \eqref{dual_rep} such that $\rho(X)\le E[XZ] - E[l^*(Z)]-\varepsilon$, we must have $l^\ast(Z)< \infty$.
For every $c\in(0,1)$, define $Z^{c}:= c +(1-c)Z$, which is likewise feasible for \eqref{dual_rep}, satisfies $l^*(Z^c)<\infty$ and is such that $(Z^c)_c$ is uniformly integrable. Assume for the moment that $l^*(Z^{c})\to l^*(Z)$ $P$-a.s. Then, by convexity, $0\le l^*(Z^c)\le (1-c)l^*(Z^c)\le l^*(Z)$.
Thus, we conclude by dominated convergence that $\rho(X) \le \lim_{c\to 0}E[XZ^c]-E[l^*(Z^c)]-\varepsilon$, which yields the reverse inequality.
The proof is finished after noticing that $l^*$ must be continuous throughout its domain. 
In fact, the domain of $l^*$ is an interval with end points denoted by $a\in \mathbb{R}_+$ and $b\in \mathbb{R}_+\cup\{+\infty\}$, and $l^*$ is continuous on $(a,b)$.
If $+\infty>b \in \text{dom}(l^*)$, let $x^n\to b$ and $\lambda^n\in (0,1)$ such that $\lambda^n\to 1$ and $x^n= \lambda^nb + (1-\lambda^n)a$.
By convexity, we have $l^*(x^n) \le \lambda^nl^*(b)+(1-\lambda^n)l^*(a)$.
This shows $\limsup_{n\to \infty} l^*(x^n)\le l^*(b)$.
We conclude by lower semicontinuity that $l^*$ is continuous at $b$ and the proof is similar if $a\in \text{dom}(l^*)$.
\end{proof}

\medskip

\begin{proof}(of Proposition \ref{lem PDE})\\
Starting from Lemma \ref{thm:rep}, we see that the r.h.s.\ of \eqref{eq:reduction} is a lower bound for  $\rho(X)$, as we are working in the completed Brownian filtration. We shall establish the opposite inequality by repeated approximation arguments.\\
STEP 1: Let $Z\in {\cal Z}$. We may assume w.l.o.g.\ that $E[l^*(Z)]<\infty$, as otherwise this $Z$ is irrelevant for the problem. Letting $Z_t:=E[Z|\mathcal{F}_t]$ and $\tau^n:=\inf\{0<t\leq T:Z_t=n\}\wedge T$, we have by optional sampling and Jensen's inequality that $-E[l^*(Z)]\leq -E[l^*(Z^{n})] $ and $E[l^*(Z^{n})]<\infty$, with $Z^n:=Z_{\tau^n}\le n$. On the other hand $E[Z^nX]\to E[ZX] $ by the martingale convergence theorem. Thus
$$ E[ZX]-E[l^*(Z)]\leq \liminf_{n\to \infty}(E[Z^nX]-E[l^*(Z^n)]).$$
Therefore 
we can assume w.l.o.g.\ that $Z$ is essentially bounded from above as well as essentially bounded away from $0$.\\%
STEP 2: Define $Z^n:=E[Z|\mathcal{G}_n]$, where $\mathcal{G}_n=\sigma\{W_{kT2^{-n}}:k=0,\dots,2^n\}$.
It holds $$\textstyle Z^n = u_n(W_{T/n},\dots,W_{kT/n} ,\dots,W_{T} ),$$ for some bounded positive Borel function $u_n$ which is bounded away from $0$ and with range in $dom(l^*)$.
Moreover, $Z^n\to Z$ $P$-a.s. and by continuity of $l^*$ in its domain (see the proof of Lemma \ref{thm:rep}), $l^*(Z^n)$ is essentially bounded uniformly in $n$. Using martingale convergence again and dominated convergence we have 
$$ E[ZX]-E[l^*(Z)] = \lim_{n\to \infty}(E[Z^nX]-E[l^*(Z^n)]).$$
Thus, we may further assume w.l.o.g.\ that $Z$ is of the form
\begin{equation}
\label{eq:z_u}
Z = u(W_{T/n},\dots,W_{kT/n} ,\dots,W_{T} )\quad \text{for some } n\in \mathbb{N},
\end{equation}
with $u$ as above.\\
STEP 3: Let $\Phi\in C^\infty_c(\mathbb{R}^n) $ be the mollifier
$$\textstyle \Phi(x)=\mathbb{I}_{\{\|x\|<1\}}\lambda\exp\left\{ \frac{1}{\|x\|^2-1}  \right\},$$
with $\lambda\ge 0$ such that $\int_{\mathbb{R}^n}\Phi(x)\,dx=1$.
Define by convolution $u^{\delta}(x):= u*\delta^{-n}\Phi(x/\delta) $, $\delta>0$.
Then $u^\delta \in C^\infty(\mathbb{R}^n)$ is bounded from above, bounded away from $0$ and the derivative $ \nabla u^\delta = u*\nabla(\delta^{-n}\Phi(\cdot/\delta)) $ is bounded. We can also choose a sequence $\delta_m\to 0$ so that $u^{\delta_m} \to u$, Lebesgue a.e.\ in $\mathbb{R}^n$; this follows from the convergence over compacts of $u^\delta$ to $u$ in $L^1(dx)$ and a diagonalization argument. This shows that $u^{\delta_m}(W_{T/n},\dots,W_{kT/n} ,\dots,W_{T} )$ converges to $Z$ almost surely as $m\to \infty$, since the law of the Gaussian vector $(W_{T/n},\dots,W_{kT/n} ,\dots,W_{T} )$ is equivalent to Lebesgue in $\mathbb{R}^n$. Arguing as in the previous step, we conclude that we may further assume w.l.o.g.\ that $Z$ is given by \eqref{eq:z_u} where $u$ is smooth and with bounded derivatives. \\
STEP 4: From the previous steps, it remains to show that for every $n$, the random variable $Z = u(W_{T/n},\dots,W_{kT/n} ,\dots,W_{T} )$ can be written as ${\cal E}(\int\beta\,dW)_T$, with $\beta \in {\cal L}_b$.
 Observe that for $t\in\left[\frac{n-1}{n}T,T\right]$ we have
\begin{align*}\textstyle
E[Z|\mathcal{F}_t]= \int u(W_{T/n},\dots,W_{T(n-1)/n},W_{t}+x)] N^{T-t}(dx) =:  U^n\left(t,W_t\,;\,W_{T/n},\dots,W_{\frac{n-1}{n}T}\right),
\end{align*}
with $N^{T-t}(dx)$ the law of a centred Gaussian with variance $(T-t)\times I_d$, with $I_d$ the identity of $\mathbb{R}^{d\times d}$. 
By the mean value theorem and dominated convergence, the function $U^n$ is differentiable in the spacial arguments and the derivatives are bounded, uniformly in the time argument. Smoothness in the time argument is apparent from the density of $N^{T-t}(dx)$. 
In addition, $U^n$ is bounded away from $0$.

We now proceed by reverse induction. Assume that we have constructed a function $U^{k+1}$ such that $E[Z|{\cal F}_t]=U^{k+1}(t,W_t\,;\,W_{T/n},\dots,W_{\frac{k}{n}T})$ on $[kT/n,(k+1)T/n]$ with $U^{k+1}$ smooth, bounded from above and away from zero, as well as having bounded derivatives  in $(x,w_1,\dots,w_{k})$ uniformly in time.  
By the tower property, for $t\in [(k-1)T/n,kT/n]$ we have:
\begin{align*}
E[Z|\mathcal{F}_t]&\textstyle = E\left [ U^{k+1}(kT/n,W_{kT/n}\,;\,W_{T/n},\dots,W_{T(k-1)/n},W_{kT/n}) |\mathcal{F}_t\right]  \\& \textstyle = \int U^{k+1}(kT/n,W_{t}+x\,;\,W_{T/n},\dots,W_{T(k-1)/n},W_{t}+x)] N^{(kT/n)-t}(dx) \\
&\textstyle =:  U^k\left(t,W_t\,;\,W_{T/n},\dots,W_{\frac{k-1}{n}T}\right).
\end{align*}
By essentially the same argument as above, $U^k$ is
 smooth in time and space arguments, bounded from above and away from zero, and has bounded derivatives in $(x,w_1,\dots,w_{k-1})$ uniformly in time.

Using It\^o's formula and by uniqueness in the martingale representation, we obtain that 
$$\textstyle \beta_t=\left[{ U^k\left(t,W_t\,;\,W_{T/n},\dots,W_{\frac{k-1}{n}T}\right)}\right ]^{-1}{\nabla_x U^k\left(t,W_t\,;\,W_{T/n},\dots,W_{\frac{k-1}{n}T}\right)},$$
on $t\in [(k-1)T/n,kT/n]$. Since there are only finitely many  such intervals for fixed $n$, it follows that  
$\beta$ is essentially bounded. This concludes the proof.
\end{proof}



\end{document}